\documentclass[12pt,a4paper]{article}
\usepackage[intlimits]{amsmath}
\usepackage{amsfonts,amssymb,amscd,amsthm}
\usepackage{cite}

\usepackage[all]{xy}

\def\lra{\longrightarrow}
\def\ind{\operatorname{ind}}

\def\Td{\operatorname{Td}}
\def\im{\operatorname{Im}}
\def\td{\operatorname{td}}
\def\tr{\operatorname{tr}}
\def\Mat{\operatorname{Mat}}
\def\End{\operatorname{End}}
\def\Vect{\operatorname{Vect}}
\def\dim{\operatorname{dim}}
\def\rk{\operatorname{rk}}
\def\Hom{\operatorname{Hom}}
\def\td{\mathrm{Todd}}
\def\Ch{\operatorname{Ch}}
\def\ch{\operatorname{ch}}
\def\Char{\operatorname{Char}}

\newtheorem{theorem}{Theorem}[section]
\newtheorem{lemma}{Lemma}[section]
\newtheorem{proposition}{Proposition}[section]

\theoremstyle{definition}

\newtheorem{remark}{Remark}[section]
\newtheorem{example}{Example}[section]
\newtheorem{definition}{Definition}[section]

\numberwithin{equation}{section}

\title{Index of elliptic operators for a diffeomorphism}
\author{Anton Savin and Boris Sternin}

\begin{document}

\maketitle

\begin{abstract}
We develop elliptic theory of operators associated with a diffeomorphism of a closed smooth manifold.
The aim of the present paper is to obtain an index formula for such operators in terms
of topological invariants of the manifold and of the symbol of the operator.  The symbol in this situation is an element of 
a certain crossed product. We express the index as the pairing of the class in K-theory defined by the symbol and the Todd class in periodic  cyclic cohomology of the crossed product. 
\end{abstract}



\tableofcontents

\section*{Introduction}
\addcontentsline{toc}{section}{Introduction}

Let $M$ be a smooth manifold and $g:M\to M$ be a diffeomorphism. We develop elliptic theory for operators of the form
\begin{equation}\label{eq-opers0}
    D=\sum_k D_kT^k:C^\infty(M)\lra C^\infty(M).
\end{equation}
Here $T$ is the shift operator $Tu(x)=u(g(x))$ along the orbits of  $g$, $D_k$ are pseudodifferential operators ($\psi$DO) on $M$, and the sum is assumed to be finite.

The aim of the present paper is to obtain an index formula for the operator   \eqref{eq-opers0} in terms
of topological invariants of the manifold and of the symbol of the operator. Precise definitions of all the objects will be given below but now let us note an important characteristic property of this theory.  Namely, the algebra of symbols of  operators 
\eqref{eq-opers0} is {\em not commutative}. More precisely, an explicit computation shows that the algebra of symbols is the crossed product   $C^\infty(S^*M)\rtimes \mathbb{Z}$ of the algebra of functions on the cosphere bundle by the action of the group $\mathbb{Z}$. This essentially means that we consider noncommutative elliptic theory.

Special  cases of operators~\eqref{eq-opers0} were considered by a number of authors (e.g., see \cite{Con1,NaSaSt17,SaSt24,SaSt25,SaSt29,Ant4,AnLe2,Per2,Per3}). In these papers, as a rule, certain conditions were imposed on the manifold $M$  and on the   diffeomorphism  $g$. For example, in the book   \cite{NaSaSt17} it is assumed that the  diffeomorphism is an isometry,  in \cite{AnLe2} the diffeomorphism is arbitrary but the manifold is one-dimensional, and so on. Further, we would   like to mention interesting papers   \cite{CoMo2} and \cite{Mos2}, where the authors study elliptic operators of the form   \eqref{eq-opers0} that are associated with the Dirac operator and conformal diffeomorphisms. Let us stress that in the present paper we consider an arbitrary compact smooth manifold and an arbitrary diffeomorphism   {\em without any restrictions.}

The main result of the paper is an explicit index formula for the elliptic operator  
 \eqref{eq-opers0}. More explicitly, the answer is given by the formula  
\begin{equation}\label{eq-f-la1}
 \ind D=(2\pi i)^{-n}\langle [\sigma(D)],{\rm Todd} (T^*_\mathbb{C}M)\rangle,\qquad \dim M=n,
\end{equation}
in terms of cyclic cohomology, where $[\sigma(D)]\in K_0(C^\infty(S^*M\times \mathbb{S}^1)\rtimes \mathbb{Z})$ is the class of symbol in $K$-theory, ${\rm Todd} (T^*_\mathbb{C}M)\in HP^{ev}(C^\infty(S^*M\times \mathbb{S}^1)\rtimes \mathbb{Z})$ is the Todd class in cyclic cohomology and the brackets $\langle,\rangle$ denote the pairing of   $K$-theory and cyclic cohomology.

Let us briefly describe the methods used in the present paper. It is clear more or less that a noncommutative elliptic theory requires a noncommutative apparatus: noncommutative differential forms, noncommutative trace, etc. Moreover, since the diffeomorphism generates an action of the group $\mathbb{Z}$,  the relevant topological invariants are naturally elements of the Haefliger cohomology group   $H^*(S^*M/\mathbb{Z})$ (see \cite{Haef1}).  In this framework, we define the Chern character and establish an important intermediate index formula (interesting in its own right)
as an integral of a Haefliger form over  $S^*M\times \mathbb{S}^1$.  After this, we reduce the obtained formula to the natural and elegant formula   \eqref{eq-f-la1}. The latter index formula can be considered as an analogue of the Atiyah--Singer formula in our situation.

We now  describe the contents of the paper. In the first section, we introduce a notion of ellipticity and prove  finiteness theorem. In the second section, we define Chern characters for crossed products, including   twisted Chern character,  and define  the Chern character of an elliptic symbol.  Section three is devoted to the solution of the equation
\begin{equation}\label{eq-chtd1}
 \ch y=\Td x,
\end{equation}
where $\Td x$ is the Todd class of a complex vector bundle $x$ and  $\ch$ is the Chern character.
The point here is that in our situation the Todd class is generally speaking undefined. However, it can be replaced by the Chern character of a bundle satisfying Eq.~\eqref{eq-chtd1}. Finally, in the fourth section we formulate an index theorem (in   Haefliger cohomology). The proof of the index theorem is given in Sections~ 5 and 6. Namely, in Sec.~5 we reduce our initial operator to a special boundary value problem on the cylinder  $M\times [0,1]$ (see \cite{Sav9,Sav10}), 
which is then reduced to a certain pseudodifferential operator on the torus of the original manifold twisted by the diffeomorphism   $g$  (cf. \cite{APS3,BBW2}). The index of the latter operator can be computed by the Atiyah--Singer formula. However, to give an index formula in terms of the original operator, we need to compare the index formula for the pseudodifferential operator on the torus and the formula announced in Sec.~4. So, the proof of the index formula for the original operator is complete, at least in the framework of Haefliger cohomology.  In Sec.~7 we interpret the index formula in Haefliger cohomology as an Atiyah--Singer formula in cyclic cohomology. Here we use equivariant characteristic classes  in cyclic cohomology   (see \cite{Gor1}). In the eighth section we give some remarks and consider an example.

This work was done at the Institute of Analysis of Leibniz University of Hannover during  our stay in the summer 2011. We are grateful to  Prof. Elmar Schrohe, in whose working group this research was done,  for attention   and excellent working conditions.  

\section{Ellipticity and finiteness theorem}

Let $M$ be a smooth closed manifold and  $g:M\to M$ be a diffeomorphism. Consider an operator of the form
\begin{equation}\label{eq-oper1}
    D=\sum_k D_kT^k: C^\infty(M,\mathbb{C}^N)\lra C^\infty(M,\mathbb{C}^N),
\end{equation}
where
$$
T: C^\infty(M )\to C^\infty(M ), \quad Tu(x)=u(g(x)),
$$
is the shift operator corresponding to $g$ and the coefficients  
$$
 D_k:C^\infty(M,\mathbb{C}^N)\lra C^\infty(M,\mathbb{C}^N)
$$
are pseudodifferential operators  ($\psi$DO) of order zero and the sum in  \eqref{eq-oper1} is finite.
Denote the principal symbols of the coefficients by  
$$
 \sigma(D_k)\in
C^\infty(S^*M,\Mat_N(\mathbb{C})).
$$
Here  $S^*M=T^*_0M/\mathbb{R}_+$ stands for the cosphere bundle with the projection  $\pi:S^*M\to M$, where $T^*_0M=T^*M\setminus 0$ is the cotangent bundle with the zero section deleted.
\begin{definition}
\emph{The symbol} of the operator $D$ is a collection $\sigma(D)=\{\sigma(D_k)\}$ of symbols of its coefficients.
\end{definition}
If
\begin{equation}
B=\sum B_kT^k:C^\infty(M,\mathbb{C}^N)\to C^\infty(M,\mathbb{C}^N)
\label{eqa1}
\end{equation}
is another operator of the form~\eqref{eq-oper1}, then the symbol of the composition of  
\eqref{eq-oper1} and  \eqref{eqa1} is determined by the formula
\begin{equation}\label{eq-compos1}
\sigma(DB)(k)=\sum_{l+m=k } \sigma(D_l)\left[(\partial
g)^{l*}\sigma(B_m)\right].
\end{equation}
The product of symbols in the right-hand side of  Eq.~ \eqref{eq-compos1} is called the crossed product of symbols. Here $\partial g=(dg^t)^{-1}:T^*M\to T^*M$ is the codifferential of  $g$.

\begin{definition}\label{def-ell-1}
An operator  $D$ is   \emph{elliptic} if there exists a  symbol   $\sigma(D)^{-1}$  with a finite number  of nonzero components such that  
$$
\sigma(D)\sigma(D)^{-1}=\mathbf{1},\quad \sigma(D)^{-1}\sigma(D)=\mathbf{1},
$$
where the product of symbols is defined by Eq.~\eqref{eq-compos1}, while 
 the symbol of the identity operator $Id=T^0$   is denoted by $\mathbf{1}$.
\end{definition}

The composition formula  \eqref{eq-compos1} readily implies the following finiteness theorem.
\begin{theorem}\label{th-finit1}
An elliptic operator \eqref{eq-oper1}  is Fredholm  in Sobolev spaces
$$
D:H^s(M,\mathbb{C}^N)\lra H^s(M,\mathbb{C}^N)
$$
for all  $s$, and its kernel and cokernel consist of smooth functions.  
\end{theorem}

\begin{proof} Indeed, since $D$ is elliptic, the inverse symbol $\sigma(D)^{-1}$ has finitely many nonzero components. Denote by $D^{-1}$ an arbitrary operator with symbol equal to $\sigma(D)^{-1}$.
Then a direct computation shows that   $D^{-1}$  is an inverse of $D$ modulo operators of negative order.
\end{proof}

\begin{remark}
The set of all symbols  is an algebra with respect to the product \eqref{eq-compos1}. This algebra is actually the algebra of matrices, whose entries are elements of the crossed product    (e.g., see \cite{Zell1})
of the algebra  $C^\infty(S^*M)$ of smooth functions on $S^*M$ and the group $\mathbb{Z}$. 
The latter algebra is denoted by   $C^\infty(S^*M)\rtimes \mathbb{Z}$.  In the present paper, we only consider algebraic crossed products, whose elements have at most a finite number of nonzero components.  
\end{remark}

\section{Chern characters for crossed products}\label{sec-3}

\subsection{Chern character}\label{sec-2-1}

Let $g:X\to X$ be a diffeomorphism of a smooth closed manifold   $X$ and $E\in \Vect(X)$ be a vector bundle.

We recall some facts of the theory of noncommutative differential forms 
(e.g., see \cite{Con1,Con8,NaSaSt17}).

\paragraph{Noncommutative differenatial forms.}

Let $\Lambda(X)$ be the algebra of differential forms on $X$ with smooth coefficients.
Following \cite{NaSaSt17}, we define the space $\Lambda(X,\End E)_{\mathbb{Z}}$ of {\em noncommutative forms} on $X$. This space consists of finite sequences
$$
 a=\{a(k)\},\qquad a(k)\in\Lambda(X)\otimes \Hom({g^k}^*E,E),\quad \deg a=\max_k \deg a(k),
$$
which we  represent as operators
$$
 a=\sum_k a(k)T^k:\Lambda(X,E)\to \Lambda(X,E),
$$
where for  $\omega\in \Lambda(X,E)$ we set $T\omega:=g^*\omega\in \Lambda(X,g^*E)$.
This operator interpretation endowes   the space   $\Lambda(X,\End E)_{\mathbb{Z}}$ with the algebra structure.
Namely, the product of two forms  $a=\sum_k a(k)T^k,b =\sum_l b(l)T^l $ is the form
$$
ab=\sum_{k,l} a(k) {g^k}^*(b(l))T^{k+l}.
$$
The subalgebra of forms of zero degree is denoted by   $C^\infty(X,\End E)_{\mathbb{Z}}$.

\begin{remark}
For a trivial bundle  $\mathbf{1}_n$ of rank $n$,  we have
$$
\Lambda (X,\End \mathbf{1}_n)_{\mathbb{Z}}\simeq \Lambda (X,\Mat_n(\mathbb{C}))\rtimes
{\mathbb{Z}}.
$$
More generally, suppose we are given a  $g$-bundle $E$. This means that the mapping   $g:X\to X$ is extended   to a fiberwise-linear mapping $\widetilde{g}=\alpha g^*:E\to E$, where $\alpha:g^*E\to E$  is an isomorphism of vector bundles.
In this case we have an isomorphism of algebras  
$$
\begin{array}{ccc}
  \Lambda(X,\End E)_\mathbb{Z} & \simeq  & \Lambda(X,\End E )\rtimes \mathbb{Z},\vspace{2mm}\\
  \sum_k a(k) T^k  & \longmapsto & \sum_k \bigl[a(k)T^k  (\alpha T)^{-k}\bigr]\widetilde{T}^k.
\end{array}
$$
Here
\begin{equation}
\widetilde{T}=\alpha T:\Lambda(X,E)\to \Lambda(X,E)
\label{eqa2}
\end{equation}
is the action of the shift operator on the sections of $E$, while $\Lambda(X,\End E )\rtimes \mathbb{Z}$ is  the 
crossed product for the shift operator \eqref{eqa2}.
\end{remark}

\paragraph{Graded trace.} 
We define a graded trace on noncommutative forms   ranging in Haefliger forms on the manifold.
To this end, we first recall necessary facts about Haefliger forms and cohomology (see \cite{Haef1}).
In the de Rham complex $(\Lambda(X),d)$, consider the subcomplex
$((1-g^*)\Lambda(X),d)$.
\begin{definition}
\emph{The space of Haefliger forms} on $X$ is the quotient space $\Lambda(X)/(1-g^*)\Lambda(X)$ and is  denoted
by $\Lambda(X/\mathbb{Z})$.
The cohomology of the quotient complex $(\Lambda(X)/(1-g^*)\Lambda(X),d)$  is  \emph{Haefliger cohomology }of $X$ with respect to the diffeomorphism $g$ and is denoted by $H(X/\mathbb{Z})$.
\end{definition}

\begin{example}
It is clear from the definition, that Haefliger forms are automatically $g$-invariant. Moreover, if $g^N=Id$, then the spectral decomposition with respect to $g$ shows that the Haefliger complex is isomorphic to the complex $(\Lambda(X)^g,d)$ of $g$-invariant forms. Therefore, Haefliger cohomology in this case is isomorphic to the cohomology of the quotient
  $X/\mathbb{Z}_N$ of $X$ by the action of the group generated by the diffeomorphism $g$. This gives an explanation for our notation for Haefliger cohomology.
\end{example}

Consider the mapping
\begin{equation}
\label{trace1}
\begin{array}{ccc}
 \tau_E:\Lambda(X,\End E)_\mathbb{Z} & \longrightarrow &  \Lambda (X/\mathbb{Z}), \vspace{1mm}\\
 \sum_k\omega(k)T^k & \longmapsto & \tr_E(\omega(0)),
\end{array}
\end{equation}
where $\Lambda(X/\mathbb{Z})$ is the space of Haefliger forms on $X$, while $\tr_E:\Lambda(X,\End E) \to \Lambda(X)$ 
is the trace of an endomorphism of $E$.

The mapping \eqref{trace1} for the trivial bundle   $E=X\times \mathbb{C}^n$  will be denoted simply by $\tau$.
\begin{lemma}
The mapping  \eqref{trace1} is a graded trace on the algebra
$\Lambda(X,\End E)_\mathbb{Z}$, i.e., 
$$
\tau_E\bigl([\omega_1,\omega_2]\bigr)=0,\quad \text{whenever }\omega_1,\omega_2\in \Lambda(X,\End E)_\mathbb{Z}
$$
where $[,]$ stands for the supercommutator
$$
[\omega_1,\omega_2]=\omega_1\omega_2-(-1)^{\deg \omega_1 \deg
\omega_2}\omega_2\omega_1.
$$
\end{lemma}

\begin{proof}
The proof is straightforward:
\begin{multline}
\tau_E(\omega_1 T^k \omega_2T^{-k})=\tr_E( \omega_1g^{k*}(\omega_2))=
\tr_{g^{-k*}E} \bigl(g^{-k*}(\omega_1g^{k*}(\omega_2))\bigr)=\\
\tr_{g^{-k*}E} ((g^{-k*}\omega_1)\omega_2)=(-1)^{\deg \omega_1 \deg \omega_2}\tau_E(\omega_2 T^{-k}\omega_1 T^k).
\end{multline}
Here $\omega_1\in\Lambda(X,\Hom(g^{k*}E,E))$ and $\omega_2\in\Lambda(X,\Hom(g^{-k*}E,E))$ . 
The first and the last equalities follow from the definition
of the product of noncommutative forms; the second equality follows from the properties of Haefliger forms; the third equality follows from the properties of the induced mapping.  
\end{proof}

\paragraph{Noncommutative connection and curvature form.}
We choose a connection in $E$
$$
\nabla_E:\Lambda (X,E)\longrightarrow \Lambda (X,E).
$$

Given a projection $p\in C^\infty(X,\End E)_\mathbb{Z}$, we define a differential operator of order one  
\begin{equation}\label{nabbla}
\nabla:=p\nabla_E p:\Lambda (X,E)\longrightarrow \Lambda (X,E).
\end{equation}
A {\em noncommutative connection} for a projection $p$ is the sum
$$
p(\nabla_E +\omega)p
$$
of operator \eqref{nabbla} and an operator of multiplication by $p\omega p$, where 
$\omega\in\Lambda^1(X,\End E)_\mathbb{Z}$.
\begin{lemma}
\label{lemA} For a noncommutative connection  $\nabla$ is one has an equality 
$$
d\tau_E(A)=\tau_E\left([\nabla,A]\right),
$$
whenever   $A\in \Lambda (X,\End E)_\mathbb{Z}$ is  such that $pA=A=Ap$.
\end{lemma}
\begin{proof}
1. Since  $\tau_E$ is a graded trace, we see that the right-hand side of the equality does not depend on the choice of   $\nabla$. Therefore, below we assume that $\nabla$ is defined as in   \eqref{nabbla}.

2. For a trivial bundle with the trivial connection  $\nabla=p \cdot d\cdot p$ we obtain
$$
\tau\left([p\cdot d\cdot p,A]\right)=\tau(pdp A+pdA-dp
A)=\tau(pdA)=\tau(dA)=d\tau(A).
$$
(Here we use the identities $(dA) p+A(dp)=dA$ and $(dp)A+pdA=dA$, which are obtained by differentiation of  $Ap=A=pA$.)

3. Let us realize a nontrivial vector bundle $E$ as a subbundle in the trivial bundle  $X\times \mathbb{C}^n$.  
Then the left-hand side of the desired equality is equal to  
$$
d\tau_E(A)=d\tau (A),
$$
whereas the right-hand side is equal to
$$
\tau \left([p\nabla_E p,A]\right)=
\tau \left([p\nabla_{\mathbb{C}^n} p,A]\right)=
\tau \left([p\cdot d\cdot p,A]\right)=d\tau (A).
$$
(In the trivial bundle  $X\times \mathbb{C}^n$ we choose a connection $\nabla_{\mathbb{C}^n}$ equal to the direct sum
of the  connection in $E$ and some connection in its orthogonal complement.)
\end{proof}

\begin{proposition}
\label{lemB}  For any noncommutative connection  $\nabla$ the operator
$$
 \nabla^2:\Lambda (X, E)\longrightarrow \Lambda (X,E)
$$
is an operator of multiplication by a $2$-form.  This form is denoted by
\begin{equation}\label{eqa3}
\Omega\in\Lambda^2(X,\End E)_\mathbb{Z}.
\end{equation}
and is called the  {curvature form of the noncommutative connection $\nabla$.}
\end{proposition}

\begin{proof}Let us embed $E$ as a subbundle of the trivial bundle $X\times\mathbb{C}^n$. 
Then the direct sum of the noncommutative connection for   $p$, and some noncommutative connection
for  $1-p$ is a noncommutative connection of the form  $\nabla_{\mathbb{C}^n}=d+\omega$, where
$\omega$ is a matrix-valued noncommutative   $1$-form on $X$.

Given a section $u=pu$, a direct computation enables us to compute the curvature form
\begin{multline*}
(p\nabla p)^2u=(p\nabla_{\mathbb{C}^n}p)^2u=(p(d+\omega)p)^2u=(p(d+\omega))^2u=\\
=(pdpdp-\omega dp+\omega p\omega+dp\omega+pd\omega)u.
\end{multline*}
\end{proof}

\paragraph{Chern character.}

\begin{definition}\label{chern1}
\emph{The Chern character form} of a projection $p\in C^\infty(X,\End E)_\mathbb{Z}$ is   the Haefliger form
$$
\ch  p:=\tau_E\left(p\exp \left(-\frac \Omega{2\pi i}\right)\right)\in
 \Lambda^{ev}(X/\mathbb{Z}),
$$
where $\Omega$ is the curvature form  \eqref{eqa3}.
\end{definition}

\begin{proposition}
The form $\ch  p$ is closed and its Haefliger cohomology class does not depend on the choice of a noncommutative connection and is determined by the class of projection  $p$ in the group $K_0(C^\infty(X,\End E)_\mathbb{Z})$.
\end{proposition}
\begin{remark}
Here   the $K_0$-group of  $C^\infty(X,\End E)_\mathbb{Z}$ is by definition the  Grothendieck group of 
homotopy classes of matrix projections with entries in this algebra.  
\end{remark}

\begin{remark}Strictly speaking, to define the Chern character on the $K$-group,  we need
to consider arbitrary matrix projections over   $C^\infty(X,\End E)_\mathbb{Z}$, while we considered only scalar projections. However, matrix projections can be considered as elements of the algebra   $C^\infty(X,\End (E\otimes \mathbb{C}^k))_\mathbb{Z}$. Therefore, we do not consider the matrix case to avoid excessively complicated notation.  
\end{remark}
\begin{proof}

1. By Lemma  \ref{lemA}, the form $\ch p$ is closed.  Indeed, we have
$$
d\tau_E(\Omega^k)=\tau_E\left([\nabla,\Omega^k]\right)=0,\quad \text{since }[\nabla,\nabla^{2k}]=0.
$$

2. Let us show that the cohomology class of   $\ch p$ does not depend on the choice of the noncommutative connection $\nabla$. Let $\nabla_{0},\nabla_{1}$ be two noncommutative connections for  $p$. Then their difference is an operator of multiplication by a noncommutative   $1$-form: $\alpha:=\nabla_{1}-\nabla_{0}$. Consider the homotopy of noncommutative connections  
$$
\nabla_{t}=(1-t)\nabla_{0}+t\nabla_{1}.
$$
Then we have $\frac d{dt}\nabla_t=p(\nabla_1-\nabla_0)p=p\alpha p.$ Hence
\begin{multline*}
\frac d{dt}\tau_E(\nabla_t^{2k})=k\tau_E\left(\left(\frac
d{dt}\nabla_t^2\right)\nabla_t^{2k-2}\right)=k\tau_E\left(\left[\nabla_t,\frac
d{dt}\nabla_t\right]\nabla_t^{2k-2}\right)=\\
=k\tau_E\left(\left[\nabla_t,\left(\frac
d{dt}\nabla_t\right)\nabla_t^{2k-2}\right]\right)=d\tau_E\left(k\left(\frac
d{dt}\nabla_t\right)\nabla_t^{2k-2}\right).
\end{multline*}
Integrating this expression over $t\in[0,1]$, we obtain:
\begin{equation}\label{eq-warum1}
\tau_E(\nabla_1^{2k})-\tau_E(\nabla_0^{2k})=d\omega',
\end{equation}
where $\omega'$ is some differential form. Equality \eqref{eq-warum1} means that the forms $\ch p$ defined in terms of two connections $\nabla_{ 1}$ and $\nabla_{ 0}$, are cohomologous.

3. Let $p_t,t\in[0,1]$ be a smooth homotopy of projections connecting $p_0$ to $p_1$.
We want to show that the difference of the corresponding Chern forms is an exact form. By item~2 of the present proof, it suffices to consider the case, where  $p$ acts in a trivial bundle with the trivial connection  $\nabla_E=d$.  In this case,
the functional  $\tau_E$ is a differential graded trace. Hence, homotopy invariance of  the Chern form in Haefliger cohomology follows from the standard computations   (e.g., see  \cite{NaSaSt17}).
\end{proof}

By this proposition, we obtain a well-defined mapping   (Chern character)
$$
 K_0(C^\infty(X,\End E)_\mathbb{Z})\stackrel{\ch }\longrightarrow   H^{ev}(X/\mathbb{Z})\otimes \mathbb{C}.
$$

 \begin{example}\label{exaa1}
Let $E=X\times\mathbb{C}^N$ be the trivial bundle and $\nabla_E=d$. Then  the noncommutative connection is equal to $\nabla_p=p d p$, its curvature form is equal to $\Omega_p=(\nabla_p)^2=pdpdp$. Hence, the Chern character form is given by the standard formula  
\begin{equation}\label{eq-ch-flat}
\ch p = \tr [p\exp(-dpdp/2\pi i)]_0,
\end{equation}
where $\omega_0$ denotes the coefficient at $T^0=1$  and $\tr$ is the matrix trace.
\end{example}

\subsection{Twisted Chern character (Chern character with coefficients in a vector bundle)}

Given a diffeomorphism $g:X\to X$  as above, we defined Chern character on the $K$-group $K_0(C^\infty(X)\rtimes \mathbb{Z})$. Suppose now, we are also given a $g$-bundle $E\in\Vect(X)$ (i.e., there is an extention of the diffeomorphism $g:X\to X$   to a fiberwise isomorphism $E\to E$). Then  on the same  $K$-group 
we can define Chern character twisted by $E$. To define this twisted Chern character, consider the algebra homomorphism
$$
\begin{array}{ccc}
 \beta: C^\infty(X)\rtimes \mathbb{Z} & \longrightarrow & C^\infty(X,\End E)\rtimes \mathbb{Z},\vspace{2mm}\\
 \sum_k a(k)T^k & \longmapsto & \sum_k (a(k)\otimes 1_E) \widetilde{T}^k,
\end{array}
$$
where $\widetilde{T}:C^\infty(X,E)\to C^\infty(X,E)$ is the shift operator of $E$ (see (\ref{eqa2})).

\begin{definition}{\em The twisted Chern character} is the composition of mappings:
$$
\ch_E: K_0(C^\infty(X)\rtimes\mathbb{Z})\stackrel{\beta_*}\lra
 K_0(C^\infty(X,\End E)\rtimes\mathbb{Z})\stackrel{\ch}\lra
H^{ev}(X/\mathbb{Z}).
$$
\end{definition}
The following proposition is obvious.
\begin{proposition}\label{prop-first-class1}
The composition of mappings  
$$
 K_0(C^\infty(X))\to K_0(C^\infty(X)\rtimes\mathbb{Z})\stackrel{\ch_E}\to H^{ev}(X/\mathbb{Z})
$$ 
is equal to $[p]\mapsto (\ch \im p)(\ch E) $, where $\im p\in\Vect(X)$ is the vector bundle defined by   projection $p$, 
while $\ch$ is the Chern character of a vector bundle.
\end{proposition}

\subsection{Chern character of an elliptic symbol}

\paragraph{Class of symbol in $K$-theory.}
Let $D$ be an elliptic operator of the form~\eqref{eq-oper1}. To its symbol $\sigma(D)$ we now assign an element in   $K$-theory. To this end, we extend the diffeomorphism   $\partial g:S^*M\to S^*M$  to a diffeomorphism  
$S^*M\times \mathbb{S}^1\to S^*M\times \mathbb{S}^1$ that acts as identity along  $\mathbb{S}^1$. 
This action defines a crossed product  denoted by  $C^\infty( S^*M\times \mathbb{S}^1)\rtimes\mathbb{Z}$. 
Consider the projection $\mathcal{P}=\{\mathcal{P}(\varphi)\}$  
$$
\mathcal{P}(\varphi)=
\left\{
\begin{array}{cl}
  \left(%
\begin{array}{cc}
  I_N\cos^2\varphi & \sigma(D)\sin\varphi\cos\varphi \\
  \sigma(D)^{-1} \sin\varphi\cos\varphi &I_N\sin^2\varphi \\
\end{array}%
\right), & \text{for  }\varphi\in[0,\pi/2],\\
 \left(%
\begin{array}{cc}
 I_N \cos^2\varphi & I_N \sin\varphi\cos\varphi \\
    I_N\sin\varphi\cos\varphi & I_N\sin^2\varphi \\
\end{array}%
\right), & \text{for }\varphi\in[ \pi/2,2\pi],
\end{array}
\right.
$$
where $\varphi$ is the coordinate on the circle. Note that the coefficients of  $\mathcal{P}$ are piecewise smooth functions of $\varphi$.

Let us define the element
\begin{equation}\label{eq-diff-2}
    [\sigma(D)]=[\mathcal{P}]\in K_0(C^\infty( S^*M\times \mathbb{S}^1)\rtimes\mathbb{Z}),
\end{equation}
where $[\mathcal{P}]$ is the equivalence class of the smoothed family of projections  $\{\mathcal{P}(\varphi)\}$ in a neighborhood of submanifolds $\varphi=0$ and $\varphi=\pi/2$.

\paragraph{Chern character.} 
Given an elliptic symbol and a   $g$-bundle $E\in \Vect(X)$, it follows from the constructions of the previous subsection that 
we have the twisted Chern character 
\begin{equation}\label{eq-ch5}
\ch_E [\sigma(D)]\in H^{ev}(S^*M/\mathbb{Z})
\end{equation}
in Haefliger cohomology. In Sec.~\ref{sec-4}, we  define one special twisting bundle useful for  the index formula. Let us now  obtain one property of the Chern character~\eqref{eq-ch5} that simplifies its computation. 

Let $t=\varphi/2\pi$ be the coordinate along on the torus   $S^*M\times\mathbb{S}^1$.
\begin{lemma}\label{lem-useful1} One has 
\begin{equation}\label{eq-s1}
 \ch_E[\sigma(D)]=\tr_E \exp\left(-\frac{\nabla_{tor}^2}{2\pi i}\right)
       \in H^*((S^*M\times\mathbb{S}^1)/\mathbb{Z}),
\end{equation}
where the operator
\begin{equation}
\nabla_{tor} =dt \frac \partial{\partial t}+t\nabla +(1-t)\sigma^{-1}\nabla \sigma,\qquad \sigma=\beta(\sigma(D)),
\label{eqa4}
\end{equation}
is defined by an arbitrary noncommutative connection  $\nabla $ in~$E$, and $\nabla_{tor}^2$ is the curvature form.
\end{lemma}

\begin{proof}
We have $[\sigma(D)]=[\mathcal{P}]$ (see  \eqref{eq-diff-2}).
Consider the isomorphism
$$
U_\varphi:\im \mathcal{P}(0) \lra \im \mathcal{P}(\varphi) ,\qquad \varphi\in[0,2\pi].
$$
$$
U_\varphi=\left\{
\begin{array}{cl}
\left(              \begin{array}{cc}
                     I_N\cos \varphi & \sigma (-\sin\varphi)\\
                      \sigma ^{-1}  \sin\varphi  & I_N\cos\varphi
                    \end{array}
                  \right), & \text{if } \varphi\in[0,\pi/2],\\
\left(
                    \begin{array}{cc}
                      \cos (\varphi-\pi/2)  & -\sin(\varphi-\pi/2)\\
                        \sin(\varphi-\pi/2) & \cos(\varphi-\pi/2)
                    \end{array}
                  \right)
\left(
                    \begin{array}{cc}
                      0 &- \sigma \\
                      \sigma^{-1}    & 0
                    \end{array}
                  \right),
& \text{if } \varphi\in[\pi/2,2\pi].
\end{array}
\right.
$$
We use this isomorphism and operator \eqref{eqa4} to 
define the noncommutative connection

\begin{equation}\label{eq-kuka1}
\nabla'=(\mathcal{P}(\varphi) U_\varphi )\nabla_{tor} (U_\varphi^{-1}\mathcal{P}(\varphi)),
\end{equation}
for the projection $\mathcal{P}=\{\mathcal{P}(\varphi)\}$ on the cylinder $S^*M\times [0,2\pi]$.  
We claim that this expression defines a connection on the torus  $S^*M\times\mathbb{S}^1$. 
To prove this, we need to check that the coefficients of the connection at $\varphi=0 $ and $\varphi=2\pi$ are compatible. We have at $\varphi=0$ 
$$
\nabla'|_{\varphi=0}=\left(\begin{array}{cc}  1 &0 \\  0  & 0   \end{array}  \right)
\sigma^{-1}\nabla  \sigma \left(\begin{array}{cc}  1 &0 \\  0  & 0   \end{array}  \right)=
\left(\begin{array}{cc}  \sigma^{-1}\nabla  \sigma  &0 \\  0  & 0   \end{array}  \right) ,
$$
while at  $\varphi=2\pi$ we obtain
$$
\nabla'|_{\varphi=2\pi}=\left(\begin{array}{cc}  1 &0 \\  0  & 0   \end{array}  \right)
\left(\begin{array}{cc} \sigma^{-1} &0 \\  0  & \sigma   \end{array}  \right)
 \nabla    \left(\begin{array}{cc} \sigma &0 \\  0  & \sigma^{-1}   \end{array}  \right)
\left(\begin{array}{cc}  1 &0 \\  0  & 0   \end{array}  \right) =\left(\begin{array}{cc}  \sigma^{-1}\nabla  \sigma  &0 \\  0  & 0   \end{array}  \right) .
$$
Therefore, we obtain the equality $\nabla'|_{\varphi=0}=\nabla'|_{\varphi=2\pi}$, i.e., the coefficients of the connection
are compatible and, therefore,  $\nabla'$ is a well-defined connection on the torus $S^*M\times\mathbb{S}^1$. 

The powers of the curvature form of this connection are  equal to 
$$
(\nabla')^{2N}=(\mathcal{P}(\varphi) U_\varphi) \nabla_{tor}^{2N} (U_\varphi^{-1}\mathcal{P}(\varphi)), \quad N\ge 1.
$$
Hence, we have $ \tau_E({\nabla'}^{2N})=\tau_E({\nabla}_{tor}^{2N}), $ 
i.e., we obtain the desired equality \eqref{eq-s1}.
\end{proof}

\section{The equation $\ch y=\Td x$}\label{sec-4}

Let $x$ be a complex vector bundle over some space $Z$. Consider the equation
\begin{equation}\label{eq-xyz}
\ch y=\Td x,
\end{equation}
where $\Td x$~ is the Todd class of $x$. Since the Chern character defines a rational isomorphism  $K^0(Z)\otimes \mathbb{Q}\simeq H^{ev}(Z)\otimes\mathbb{Q}$,
Eq.~\eqref{eq-xyz} has a unique solution $y \in K^0(Z)\otimes \mathbb{Q}$, which we denote for brevity by   $\psi(x)$. Moreover, the mapping  $x\mapsto \psi(x)$  defines an operation
$$
\psi: K^0(Z)\otimes \mathbb{Q}\lra K^0(Z)\otimes \mathbb{Q}
$$
in $K$-theory with rational coefficients. This operation is multiplicative:
$$
\psi(a+b)=\psi(a)\psi(b)
$$
(this follows from the multiplicative property of the Todd class) and stable:
$$
\psi(a+1)=\psi(a).
$$

According to a theorem of Atiyah \cite{Ati2} any stable operation in $K$-theory is a formal power series in Grothendieck  operations    $\gamma_j$, $j=1,2,...$ with rational coefficients. In addition, any multiplicative operation is determined by a formal power series    
$$
 f(x)=1+\sum_{k\ge 1}a_kx^k
$$  as follows:
\begin{enumerate}
\item[1)] the infinite product  $f(x_1)f(x_2)...$ is represented as a symmetric formal power series in variables
 $x_1,x_2,...$. Hence, this product is expressed as a formal power series in terms of elementary symmetric functions  $\sigma_1(x_1,x_2,...),\sigma_2(x_1,x_2,...),...$
$$
\prod_j f(x_j)=P(\sigma_1,\sigma_2,...);
$$
\item[2)] now a multiplicative operation for   $f$ is obtained if we replace the elementary symmetric functions by Grothendieck  operations 
$$
P(\gamma_1,\gamma_2,...).
$$
\end{enumerate}

For the operation $\psi$, the corresponding formal power series is computed in the following proposition.

\begin{proposition}
The multiplicative operation $\psi$ is defined by the series
$$
\psi(x)=\frac{\ln(1+x)}x(1+x)=1+\sum_{n=1}^\infty \frac{(-1)^{n+1}}{n(n+1)}x^n.
$$
\end{proposition}
\begin{proof}
To compute the coefficients of the desired series $f$, let us take $Z=\mathbb{CP}^N$
as $N\to\infty$. We have $K(\mathbb{CP}^N)=\mathbb{Z}[x]/\{x^{N+1}=0\}$, $x=[\varepsilon]-[\mathbf{1}]$,
where $\varepsilon$ is the tautological  line bundle over the projective space, and  
$\mathbf{1}$ is the  trivial line bundle. We have $\ch x= e^u-1$, where $u=[\mathbb{CP}^1]\subset H^2(\mathbb{CP}^N)=\mathbb{Z}$ is the generator.

Let us use the method of undetermined coefficients. Let
$$
 \psi(x)= \sum_{k\ge 0} c_k x^k,\qquad c_0=1.
$$
Then the equation  $\ch\psi(x)=\Td x$ is written as
$$
\sum_k c_k(e^u-1)^k=\frac{u}{1-e^{-u}}.
$$
Changing the variable by the rule $e^u-1=t$, this gives the desired function
$$
\psi(t)=(1+t)\frac{\ln(1+t)}t.
$$
\end{proof}

Note that Grothendieck  operations can be expressed in terms of operations of direct sum, tensor product and exterior powers.  Therefore,  if   $E$ is a  $g$-bundle, then 
 $\psi(E)$  (as a virtual bundle with rational coefficients) can also be considered as a   $g$-bundle. A direct computation gives the following explicit expressions for the operation  $\psi$ on spaces $Z$
of small dimension.

\begin{proposition}
The operation  $\psi$ is equal to  (here $n=\rk E$)
$$
\begin{array}{cl}
 \dim Z\le 3: &   \displaystyle\psi(E)=1+\frac {E-n}2;\vspace{2mm} \\
 \dim Z\le 5 : & \displaystyle\psi(E)= \frac{3n^2-19n+24}{24}+
                           \frac{(-3n+13)}{12}E-\frac 1 6 {E\otimes E} +\frac 7{12}\Lambda^2 E.
\end{array}
$$
\end{proposition}
\begin{proof}
The series $\psi(x)$  defines the symmetric formal power series
$$
 \prod_j (1+x_j)\prod_j\frac{\ln(1+x_j)}{x_j}\equiv AB, \quad A=1+\sigma_1+\sigma_2+\ldots.
$$
Let us express the term  $B$ in terms of elementary symmetric functions. We have
\begin{multline*}
B=\prod\left(1-\frac{x_j}2+\frac{x_j^2}3-\frac{x_j^3}4+\ldots\right)=\\
=1-\frac 1 2 \sum x_j+\frac 1 3\sum x_j^2+
\frac 1 4 \sum_{i<j}x_i x_j-\frac 1 4 \sum x_j^3-\frac 1 6\sum_{i\ne j}x_ix_j^2-\frac 1  8\sum_{i<j<k} x_i x_j x_k+\ldots,
\end{multline*}
where dots stand for terms of orders $\ge 4$. Let $p_k=\sum_j x_j^k $  and continue the computation
\begin{multline*}
B=1-\frac {p_1}2+\frac{p_2} 3+\frac{\sigma_2}4 -\frac{p_3}{12}-\frac{p_1 p_2}6-\frac{\sigma_3}8+\ldots=\\
 = 1-\frac {\sigma_1}2+\frac{\sigma_1^2-2\sigma_2} 3+\frac{\sigma_2}4 -\frac{\sigma_1^3-3\sigma_1\sigma_2+3\sigma_3}{12}-\frac{\sigma_1 (\sigma_1^2-2\sigma_2)}6-\frac{\sigma_3}8+\ldots=\\
=1- \frac {\sigma_1}2+\frac{4\sigma_1^2-5\sigma_2}{12}+
 \frac{-6\sigma_1^3+14\sigma_1\sigma_2-9\sigma_3}{24} +\ldots
\end{multline*}
Here we used Newton's formulas:
$$
p_1=\sigma_1, p_2=\sigma_1^2-2\sigma_2, p_3=\sigma_1^3-3\sigma_1\sigma_2+3\sigma_3.
$$
This implies the following expression for   $\psi$ in terms of Grothendieck  operations:
\begin{multline}\label{eq-fishka1}
\psi=(1+\gamma_1+\gamma_2+\gamma_3+\ldots)\left(1- \frac {\gamma_1}2+\frac{4\gamma_1^2-5\gamma_2}{12}+
 \frac{-6\gamma_1^3+14\gamma_1\gamma_2-9\gamma_3}{24} +\ldots\right)=\\
=1+ \frac {\gamma_1}2+\frac{-2\gamma_1^2+7\gamma_2}{12}+
 \frac{2\gamma_1^3-8\gamma_1\gamma_2+15\gamma_3}{24} +\ldots
\end{multline}

The operations  $\gamma_j$ can be expressed in terms of exterior powers  
(see \cite{Ati2})
\begin{multline*}
\gamma_t\equiv\sum{\gamma_j}t^j =\sum_{k\ge 0}\frac{t^k}{(1-t)^k}\Lambda^k(1-t)^n=\\
=(1-t)^n\left( 1+\frac t{1-t}\Lambda^1+\frac {t^2}{(1-t)^2}\Lambda^2+\frac {t^3}{(1-t)^3}\Lambda^3+\ldots\right)=\\
=(1-t)^n(1+t\Lambda^1+t^2(\Lambda^1+\Lambda^2)+t^3(\Lambda^1+2\Lambda^2+\Lambda^3)+\ldots)=\\
= 1+t(\Lambda^1 -n)+t^2\left(\Lambda^1+\Lambda^2-n\Lambda^1+\frac{n(n-1)}2 \right)+\\
+t^3\left( \Lambda^1+2\Lambda^2+\Lambda^3-n\Lambda^1-n\Lambda^2+\frac{n(n-1)}2\Lambda^1
-\frac{n(n-1)(n-2)}6\right)+\ldots.
\end{multline*}
Substituting these expressions in \eqref{eq-fishka1}, we obtain the following formula  
\begin{multline}\label{eq-fishka2}
\psi= 1+ \frac {\Lambda^1-n}2+\frac{-2(\Lambda^1-n)^2+7(\Lambda^1+\Lambda^2-n\Lambda^1+n(n-1)/2)}{12}+\\
+ \frac{2(\Lambda^1-n)^3-8(\Lambda^1-n)(\Lambda^1+\Lambda^2-n\Lambda^1+n(n-1)/2))}{24}+\\
+\frac {15(\Lambda^1+2\Lambda^2+\Lambda^3-n\Lambda^1-n\Lambda^2+\frac{n(n-1)\Lambda^1}2-
\frac{n(n-1)(n-2)}6)}{24} +\ldots.
\end{multline}
In Borel--Hirzebruch formalism \cite{BoHi1} we have $\ch =\sum e^{x_j}$ . Hence
$$
 \ch\gamma_t =\prod (1+t(e^{x_j}-1)).
$$
It follows from this expression that for a vector bundle $E\in \Vect(Z)$ we get  $\ch\gamma_k(E)\in H^{\ge 2k }(Z)$.
Now consider the class $\ch\gamma_K(E)$, where $K=(k_1,...,k_p)$ is a multiindex, and $\gamma_K(E)=\gamma_{k_1}(E)...\gamma_{k_p}(E)$. Then this expression is identically zero whenever $2|K|=2\sum k_j>\dim Z$.   
This implies that the terms denoted by dots in Eq.~\eqref{eq-fishka2} are actually equal to zero  provided that   $\dim Z\le 7$. Using this remark, we obtain the following expressions for the operation  $\psi$:

$\dim Z\le 3$: $\psi(E)=1+(E-n)/2$.

$\dim Z\le 5$:
\begin{multline}\label{eq-fishka3}
\psi(E)= 1+ \frac {E-n}2+\frac{-2(E^2-2nE+n^2)+7(E+\Lambda^2E-nE+n(n-1)/2)}{12} =\\
= \frac{3n^2-19n+24}{24}+
                           \frac{(-3n+13)}{12}E-\frac 1 6 {E\otimes E} +\frac 7{12}\Lambda^2 E
\end{multline}
and so on.
\end{proof}

\section{Index theorem}

The complexification of the cotangent bundle will be denoted by   $T^*_\mathbb{C} M=T^*M\otimes\mathbb{C}$ .

\begin{theorem}\label{th-index1}
Let $D$ be an elliptic operator.  Then its index is equal to
\begin{equation}\label{aaa}
\ind D=\int_{S^*M\times \mathbb{S}^1}\ch_{\psi(T^*_\mathbb{C}M)}[\sigma(D)],
\end{equation}
where the operation $\psi:K^0(M)\to K^0(M)\otimes \mathbb{Q}$ was defined in Sec.~\ref{sec-4}, and
$\ch$ stands for the twisted Chern character defined in Sec.~\ref{sec-3}.
\end{theorem}
The right-hand side of Eq.~\eqref{aaa} will be referred to as the  \emph{topological index} of $D$
and denoted by $\ind_{top} D$.

\begin{proposition}
For an elliptic $\psi$DO $D$,  the topological index is equal to the topological index of Atiyah and Singer  (see~\cite{AtSi0}).
\end{proposition}
\begin{proof}
In our situation we have equalities  
$$
\ch_{\psi(T^*_\mathbb{C}M)}[\sigma(D)]=\ch[\sigma(D)]\ch\psi(T^*_\mathbb{C}M)=
\ch[\sigma(D)]\Td(T^*_\mathbb{C}M).
$$
Here the second equality follows from Proposition~\ref{prop-first-class1} and the fact that  $D$ is a $\psi$DO.
The second equality follows from the   definition of     $\psi$. These equalities show that the topological index in this case is equal to  
$$
\ind_{top} D=\int_{S^*M\times \mathbb{S}^1} \ch[\sigma(D)]\Td(T^*_\mathbb{C}M).
$$
The last expression is actually the Atiyah--Singer index formula for the index of a pseudodifferential operator $D$.
\end{proof}

Theorem  \ref{th-index1} will be proved in subsequent sections. Here we give the scheme of the proof.  
\begin{enumerate}
\item First in Sec.~\ref{sec-5} we give a reduction of the operator (\ref{eq-oper1}) to some special two-term operator,  whose index is equal to the index of a certain elliptic $\psi$DO on a special smooth closed manifold: the mapping torus of the diffeomorphism  $g:M\to M$.
\item Then we prove in Sec.~\ref{sec-6} that the index of this $\psi$DO on the torus (computed using the Atiyah--Singer index formula) is equal to the topological index of the  operator   (\ref{eq-oper1}) on $M$.
\end{enumerate}

\section{Reduction of a noncommutative operator to a $\psi$DO on a closed manifold}\label{sec-5}

\subsection{Reduction to a special two-term operators}\label{sec-2-term}

In this subsection we obtain a reduction (stable homotopy) of the operator (\ref{eq-oper1}) 
to an operator of the same type, but of a simpler form.

Given matrix projections  
$$
 p,q\in C^\infty(S^*M,\Mat_N(\mathbb{C})),\quad p^2=p,q^2=q,
$$
over $S^*M$, we choose some $\psi$DOs with symbols equal to $p,q$ and denote them by $P,Q$.
\begin{definition}\label{def-1}A \emph{special two-term operator} is an operator of the form
\begin{equation}\label{eq-kvese2}
D=QD_0TP+(1-Q)D_1(1-P): H^s(M,\mathbb{C}^N)\lra H^s(M,\mathbb{C}^N),
\end{equation}
where $H^s$ is a Sobolev space, while
$$
D_0:H^s(M,\mathbb{C}^N)\lra H^s(M,\mathbb{C}^N),\quad D_1:H^s(M,\mathbb{C}^N)\lra H^s(M,\mathbb{C}^N),
$$
are $\psi$DOs of order zero such that their symbols define  vector bundle isomorphisms  
\begin{equation}\label{eq-sigmaA}
  \sigma(D_0): (\partial g)^*\im p  \lra \im q, \qquad
  \sigma(D_1): \im (1-p)  \lra \im (1-q), \\
\end{equation}
over $S^*M$. These vector bundles are defined as the ranges of projections  $p,q,1-p,1-q$.
\end{definition}
A special two-term operator is elliptic in the sense of definition~\ref{def-ell-1}. 
Moreover, an almost-inverse operator can be defined by the formula  
$$
D^{-1}=PT^{-1}D_0^{-1}Q+(1-P)D_1^{-1}(1-Q).
$$

A \emph{homotopy} of elliptic operators is a family    $\{D_t\},$ $t\in [0,1]$
of elliptic operators such that the families of their coefficients are piecewise smooth functions of the parameter   $t$  and the number  of nonzero components of the family and   its almost inverse family are uniformly bounded. Two elliptic operators  are  \emph{stably homotopic}  if there exists a homotopy between their direct sums with identity operators acting in sections of some bundles.  

\begin{proposition}{(cf. \cite{PiVo1,AnLe1})}\label{homot} The following statements hold:

1. An arbitrary elliptic operator is stably homotopic to some special two-term operator. 

2. An arbitrary special two-term operator can be reduced  by a stable homotopy and direct sum with operator 
$T\oplus T\oplus ...\oplus T$ to a direct sum of an elliptic $\psi$DO and a special two-term operator of the form  
\begin{equation}\label{eq-kvese2z}
D=PD_0TP+(1-P): H^s(M,\mathbb{C}^N)\lra H^s(M,\mathbb{C}^N),
\end{equation}
i.e., in \eqref{eq-kvese2} one can suppose that  $Q=P$ and $D_1=1.$
\end{proposition}
\begin{proof}
1. Indeed, a direct computation shows that the homotopy defined in the paper   \cite{SaSt25}
gives the desired result, i.e.,  the homotopy   preserves ellipticity and we obtain a special two-term operator at the end of the homotopy.   

2. By  \eqref{eq-sigmaA}, we have the vector bundle isomorphism
$$
\im(1-p)\simeq \im(1-q).
$$
This implies that $[\im p]=[\im q]\in K(S^*M)$.  If the ranks of the projections are large enough   (this can be achieved by a direct sum of the special two-term operator and some operator of the form  $T\oplus T\oplus ...\oplus T$), then there exists a vector bundle isomorphism  $a:\im q\to \im p$. Consider an elliptic $\psi$DO  $D'$ with the symbol
$$
\sigma(D')=a\oplus (\sigma(D_1))^{-1}: \im q\oplus\im(1-q)\lra \im p\oplus \im(1-p).
$$
Then we obtain the factorization
$$
D=D_0TP+D_1(1-P)= (D')^{-1}(D' D_0TP+(1-P)) 
$$
modulo compact operators. This proves the proposition, since a composition of operators is stably homotopic to their direct sum.  
\end{proof}

\subsection{Reduction to a boundary value problem}

Let us consider the elliptic special two-term operator 
\begin{equation}\label{eq-2-term1}
 D=D_0TP+(1_N-P):C^\infty(M,\mathbb{C}^N)\lra C^\infty(M,\mathbb{C}^N).
\end{equation}
Recall that the ellipticity condition in this case means that the symbol of $D_0$ defines an isomorphism
\begin{equation}\label{ell-sim1}
\sigma(D_0):(\partial g)^*\im \sigma(P)\longrightarrow \im \sigma(P)
\end{equation}
of vector bundles over $S^*M$. Here the vector bundles are defined by the symbol of  $P$.

On the cylinder $M\times[0,1]$ with coordinates $x$ and  $t$   consider the boundary value problem 
(see \cite{Sav9,Sav10})
\begin{equation}\label{eq-bvp1}
\left\{\begin{array}{ll}
\displaystyle \left(\frac\partial{\partial t}+(2P-1_N)\sqrt{\Delta_M}\right)u=f_1,
& u\in H^s(M\times[0,1],\mathbb{C}^N)\vspace{2mm}\\
\displaystyle D_0TPu|_{t=0}-u|_{t=1}=f_2,
&f_1\in  H^{s-1}(M\times[0,1],\mathbb{C}^N),\;\;
f_2\in  H^{s-1/2}(M,\mathbb{C}^N),
\end{array}
\right.
\end{equation}
where $\Delta_M$ is the Laplace operator defined by a metric on $M$.
This boundary value problem, denoted for brevity by 
$(\mathcal{D}, {B})$, is elliptic and one has (in {\em op.cit.})
\begin{equation}\label{eq-duha1}
\ind D = \ind (\mathcal{D}, {B}).
\end{equation}

\subsection{Homotopy of the boundary condition}

Methods of the theory of boundary value problems (e.g., see \cite{Hor3,SaScS4}) enable one to simplify the boundary operator in Eq.~\eqref{eq-bvp1} using homotopies of elliptic boundary value problems.

Namely, we start with the {\em rotation homotopy}
$$
P(\varphi)=\left(
                    \begin{array}{cc}
                      (\cos^2\varphi) P & (\cos\varphi\sin\varphi){g^{-1}}^*( D_0^{-1}P)\\
                       (\cos\varphi\sin\varphi) {g^{-1}}^*( PD_0) & (\sin^2\varphi) {g^{-1}}^*(P)
                    \end{array}
                  \right),\qquad \varphi\in [0,\pi/2],
$$
connecting the almost projections  $P\oplus 0$ and  $0\oplus {g^{-1}}^*(P)$.  
For all $\varphi\in [0,\pi/2]$  the operator $P(\varphi)$ is an almost-projection, i.e., its symbol is a projection. To check this property, it is useful to represent this homotopy in the form  
\begin{equation}\label{eq-unit1}
P(\varphi)=U_\varphi
                               \left(
                    \begin{array}{cc}
                        P &0\\
                       0 &0
                    \end{array}
                  \right) U_\varphi^{-1}
\end{equation}
in terms of the family of almost-invertible operators
$$
U_\varphi=\left(
                    \begin{array}{cc}
                     ( \cos \varphi) P & (-\sin\varphi){g^{-1}}^*( D_0^{-1}P)\\
                       ( \sin\varphi) {g^{-1}}^*( PD_0) & (\cos\varphi) {g^{-1}}^*(P)
                    \end{array}
                  \right)+\left(
                    \begin{array}{cc}
                        1_N-P &0\\
                       0 & 1_N-{g^{-1}}^*(P)
                    \end{array}
                  \right).
$$
Here we have an equality  
$ U_\varphi^{-1}=U_{-\varphi}$ modulo compact operators.

Then we define the homotopy of operators
$$
D(\varphi)=\left(
                    \begin{array}{cc}
                      (\cos \varphi) D_0g^*(P)& (\sin\varphi) P\\
                      0 & 0
                    \end{array}
                  \right)=\left(
                    \begin{array}{cc}
                       D_0g^*(P)&0 \\
                      0 & 0
                    \end{array}
                  \right)g^*(U_\varphi^{-1}).
$$
Finally, we define the homotopy of boundary value problems on the cylinder 
\begin{equation}\label{eq-bvp2}
\left\{\begin{array}{l}
\displaystyle\left(\frac\partial{\partial t}+[2P(\varphi\chi(t))-1_{2N}]\sqrt{\Delta_M}\right)U=F_1,
\vspace{2mm}\\
\displaystyle D(\varphi)TU|_{t=0}-U|_{t=1}=F_2,
\end{array}
\right.
\end{equation}
and denote this homotopy by $(\mathcal{D}_\varphi, {B}_\varphi)$.
Here the unknown function and the right-hand sides belong to the spaces  
$U\in  H^s(M\times[0,1],\mathbb{C}^{2N}) ,$  $  F_1\in  H^{s-1}(M\times[0,1],\mathbb{C}^{2N}),$
$F_2\in  H^{s-1/2}(M,\mathbb{C}^{2N})$, and
$\chi(t)$ is a smooth nonincreasing function equal to $1$ if $t\le 1/3$ and  $0$ if $t\ge 2/3$.

\begin{lemma}\label{lemma-alpha1}
The homotopy \eqref{eq-bvp2} consists of elliptic boundary value problems.
\end{lemma}
\begin{proof}
(cf. \cite{Sav9}).
The boundary condition in \eqref{eq-bvp2} relates the values of $U$ at $t=0$ and  $t=1$,
i.e., it is a nonlocal condition. Nonlocal boundary value problems of this type were considered in  \cite{SaSt10}. Let us show that this problem is elliptic in the sense of the cited paper.  Indeed, let us reduce the nonlocal problem to a local problem in a neighborhood of the boundary of the cylinder.  To this end we introduce the following unknown functions   $V$ and $W$:
\begin{equation}\label{eq-17}
    V(t)=TU(t),\quad  W(t)=U(1-t).
\end{equation}

In a neighborhood of the boundary  the system \eqref{eq-bvp2} is written in the following equivalent form
\begin{equation}\label{eq-bvp-loc}
 \left\{%
\begin{array}{l}
   T\left(\displaystyle \frac{\partial}{\partial t}+
   [2P(\varphi\chi(t))-1_{2N}]\sqrt{\Delta_M}\right)T^{-1} V(t)=TF_1(t), \vspace{2mm}\\
   \left(\displaystyle -\frac{\partial}{\partial t}+[2P(0)-1_{2N}]\sqrt{\Delta_M}\right) W(t)= F_1(1-t),\vspace{2mm} \\
  D(\varphi)V|_{t=0}-W|_{t=0}=F_2. \\
\end{array}%
\right.\quad 0\le t<1/2
\end{equation}
Note that the system \eqref{eq-bvp-loc} is already local. The first two equations of the system are elliptic. Let us show that the boundary value problem    (\ref{eq-bvp-loc}) is elliptic, i.e., it satisfies the Shapiro--Lopatinskii condition  (e.g., see \cite{Hor3}). To prove this, we consider the Calderon bundle   \cite{Hor3} (see also \cite{SaSt10})
$$
 L_+\subset S^*M\times \mathbb{C}^{4N}
$$
of the main operator in~\eqref{eq-bvp-loc}. A direct computation shows that this bundle is equal to  
\begin{multline}\label{eq-calderon}
    L_+=\im \bigl[(\partial g)^*\sigma(P(\varphi))\bigr]\oplus \im\bigl[1_{2N}-\sigma(P(0))\bigr]=\\
\im \bigl[(\partial g)^* \sigma(U_\varphi P(0))\bigr]\oplus \im\bigl[1_{2N}-\sigma(P(0))\bigr].
\end{multline}
Here the second equality follows from   \eqref{eq-unit1}.

The Shapiro--Lopatinskii condition for the problem~\eqref{eq-bvp-loc} is equivalent to the requirement that the symbol of the boundary operator, i.e., the vector bundle homomorphism  
\begin{equation}\label{eq-shalo1}
\begin{array}{ccc}
  L_+ &  \lra & S^*M\times\mathbb{C}^{2N}, \vspace{1mm}\\
  (V,W) & \longmapsto & \sigma(D(\varphi))V-W  \\
\end{array}
\end{equation}
is an isomorphism. This requirement is satisfied in our case, since   \eqref{eq-calderon} implies that
$$
 W\in \im [1_{2N}-\sigma(P(0))],
$$
and the mapping
$$
\sigma(D(\varphi))=\sigma\left[\left(
                    \begin{array}{cc}
                       D_0g^*(P)&0 \\
                      0 & 0
                    \end{array}
                  \right)g^*(U_\varphi^{-1})\right]:
\im \bigl[(\partial g)^* \sigma( U_\varphi P(0)  )\bigr]\longrightarrow \im\bigl[\sigma (P(0))\bigr]
$$
is an isomorphism of vector bundles by \eqref{eq-unit1} and \eqref{ell-sim1}.

So, the Shapiro--Lopatinskii condition is satisfied and the problem~\eqref{eq-bvp-loc} is elliptic. 
Hence, \eqref{eq-bvp2} defines a Fredholm operator.
\end{proof}

It follows from Lemma~\ref{lemma-alpha1}  and Eq.~\eqref{eq-duha1} that
$$
\ind D=\ind (\mathcal{D}_0, {B}_0)=\ind(\mathcal{D}_{\pi/2}, {B}_{\pi/2}).
$$

Let us now consider the boundary value problem  
\begin{equation}\label{eq-bvp3}
\left\{\begin{array}{l}
\displaystyle\left(\frac\partial{\partial t}+[2P(\frac\pi 2\chi(t))-1_{2N}]\sqrt{\Delta_{M,h(t)}}\right)U=F_1,\vspace{2mm}\\
\displaystyle \left( \begin{array}{cc} 0&1 \\ 1& 0 \end{array} \right)
                       TU|_{t=0}-U|_{t=1}=F_2,
\end{array}
\right.
\end{equation}
where  $\Delta_{M,h(t)}$ is the Laplace operator on $M$ for  a family of metrics $h(t)$ smoothly depending on  $t$.

\begin{lemma}The problem \eqref{eq-bvp3} is elliptic and its index is equal to the index of the problem  
$(\mathcal{D}_{\pi/2}, {B}_{\pi/2})$.
\end{lemma}
\begin{proof}
It suffices to show that the linear homotopy connecting these two boundary value problems preserves ellipticity.  

First, the  linear homotopy between the main operators of \eqref{eq-bvp3} and \eqref{eq-bvp2}   consists of elliptic operators. This follows from the fact that the operators differ  only by metrics  defining  the  Laplace operators.  

Second, for the linear homotopy between the problems   \eqref{eq-bvp3} and \eqref{eq-bvp2},
the Calderon bundle $L_+$ is constant. Moreover, for the boundary value problems in this homotopy the corresponding family of vector bundle homomorphisms   \eqref{eq-shalo1} also does not change. Hence, the linear homotopy consists of elliptic problems.   
\end{proof}

\subsection{Reduction  to a $\psi$DO on the torus}

Let $M\times_g \mathbb{S}^1 $ be the   {\em torus of the diffeomorphism} $g$. Recal that the torus of a diffeomorphism is a closed smooth manifold obtained from the cylinder   $M\times [0,1]$ by identifying its bases with a  ``twist'' defined by the diffeomorphism  $g$:
\begin{equation}\label{eq-torus}
    M\times_g \mathbb{S}^1=M\times [0,1]\Bigl.\Bigr/
       \bigl\{(x,0)\sim (g^{-1}(x),1) \bigr\}.
\end{equation}
Consider the family of metrics $h(t)$ in \eqref{eq-bvp3} such that one has  
$$
h(t)=\left\{
\begin{array}{cc}
  h, & \text{if  } t<1/3,\\
  g^*h, & \text{if  } t>2/3, \\
\end{array}%
  \right.
$$
where $h$ is some fixed metric. This family is a smooth family of metrics in the fibers of the bundle   $M\times_g\mathbb{S}^1$.

The problem \eqref{eq-bvp3} defined  by this family of metrics is denoted by   $(\mathcal{D}',\mathcal{B}')$.
The operator $\mathcal{D}'$ defines an elliptic $\psi$DO on the torus $M\times_g \mathbb{S}^1$:
\begin{equation}\label{eq-op-torus1}
\mathcal{D}'_0=\displaystyle \frac\partial{\partial t}+[2P(\frac\pi 2\chi(t))-1_{2N}]\sqrt{\Delta_{M,h(t)}}:
C^\infty(M\times_g\mathbb{S}^1,\mathcal{E})\lra C^\infty(M\times_g\mathbb{S}^1,\mathcal{E}),
\end{equation}
where $\mathcal{E}\in\Vect(M\times_g\mathbb{S}^1)$ stands for the vector bundle with the total space  
\begin{equation}\label{eq-glue1}
\mathcal{E}=(M\times[0,1]\times \mathbb{C}^{2N})/\{(x,0,v_1,v_2)\sim (g^{-1}(x),1,v_2,v_1)\}.
\end{equation}
\begin{proposition}\label{lem-tor1}
One has $\ind(\mathcal{D}',\mathcal{B}')=\ind \mathcal{D}'_0$.
\end{proposition}
\begin{proof}
1. Since the operator of boundary condition in $(\mathcal{D}',\mathcal{B}')$ is surjective, we see that the  index of the boudnary value problem is equal to the index of the same boundary value problem but with homogeneous  boundary condition. This condition has the form  
$$
\left( \begin{array}{cc} 0&1 \\ 1& 0 \end{array} \right)TU|_{t=0}-U|_{t=1}=0,
$$
i.e., it coincides with the condition of continuity of the function   $U(t)$, considered as a section of the bundle   $\mathcal{E}$ over the torus $M\times_g\mathbb{S}^1$.

2. Since the boundary condition is actually the continuity condition,  the remaining part of the proof is standard and we omit it   (e.g., see \cite{Sav9}).
\end{proof}

\section{Comparison of topological indices}\label{sec-6}

\subsection{Computation of the index of $\psi$DO\! on the torus using Atiyah--Singer formula}

The operator   $\mathcal{D}'_0$ is an operator of the form
$$
\mathcal{D}'_0=\frac\partial{\partial t}+{A}(t),
$$
i.e., it is defined by a family $\mathcal{A}=\{A(t)\}$, $t\in [0,1]$ of operators on the sections  $M\times \{t\}$ of the torus.
Moreover, the family consists of elliptic operators and the corresponding family of symbols  
$$
 \sigma(A(t))(x,\xi)=[2\sigma(P(t))(x,\xi)-1]|\xi|_t
$$
(here $|\xi|_t$ is the norm of a covector with respect to a family of metrics   $h(t)$)
has a real spectrum at each point   $(x,\xi)$. It follows that the positive spectral subspace of the symbol   
$\sigma(A(t))(x,\xi)$ (by definition this subspace is generated by the eigenvectors with positive eigenvalues) is just the space $\im \sigma(P(t))(x,\xi)$. Hence, the family of positive spectral subspaces defines a smooth vector bundle over the torus  $S^*M\times_g\mathbb{S}^1$. Denote this vector bundle by
\begin{equation}\label{eq-bundle1}
  \sigma_+(\mathcal{A} )\in\Vect(S^*M\times_g \mathbb{S}^1).
\end{equation}
The following lemma (cf. Theorem~7.4 in \cite{APS3})  expresses the index of operator $ \mathcal{D}'_0$
in terms of the  bundle \eqref{eq-bundle1}.
\begin{lemma}\label{lem-torus1}
One has
\begin{equation}\label{eq-ind3}
 \ind \mathcal{D}'_0=\int_{S^*M\times_g\mathbb{S}^1} \ch[ \sigma_+(\mathcal{A})]\Td(T^*_{\mathbb{C}}M\times_g \mathbb{S}^1).
\end{equation}
\end{lemma}
\begin{proof}
To make the paper self-contained, we give the proof of this fact.  

1. The Atiyah--Singer index formula for   $\mathcal{D}'_0$ has the form
\begin{equation}\label{eq-atsi5}
  \ind \mathcal{D}'_0=\int_{T(M\times_g\mathbb{S}^1)}\ch{[}\sigma (\mathcal{D}'_0){]}\Td(T_{\mathbb{C}}(M\times_g\mathbb{S}^1)).
\end{equation}
Here and below in the proof we identify the tangent and contagent bundles using some metric.  

2. We have the  decomposition $T(M\times_g\mathbb{S}^1)=(TM\times_g\mathbb{S}^1)\oplus\mathbf{1}$ 
into directions perpendicular and parallel to the generator of the torus. Thus, we get
$$
\Td(T_{\mathbb{C}} (M\times_g \mathbb{S}^1))=\Td(T_{\mathbb{C}}M\times_g \mathbb{S}^1).
$$

3. Denote the composition of embeddings 
$SM\times_g\mathbb{S}^1\subset TM\times_g\mathbb{S}^1\subset T(M\times_g\mathbb{S}^1) $ by
$i$. The normal bundle of this embedding is, obviously, a direct sum of two one-dimensional trivial bundles. 
Moreover, one has  
\begin{equation}\label{eq-image1}
{[}\sigma (\mathcal{D}'_0){]}=i_![ \sigma_+(\mathcal{A} )]\in K^0(T(M\times_g\mathbb{S}^1)),
\end{equation}
where
$$
i_!:K^0(SM\times_g\mathbb{S}^1)\to K^0(T(M\times_g\mathbb{S}^1))
$$
is the direct image mapping corresponding to the embedding  $i$ (see \cite{Ati2}). Applying the Riemann--Roch--Atiyah--Hirzebruch formula \cite{AtHi1} to \eqref{eq-image1}, we obtain  
$$
\ch [\sigma (\mathcal{D}'_0)]=\ch i_![ \sigma_+(\mathcal{A} )]=i_* \ch[ \sigma_+(\mathcal{A} )]
$$
(the normal bundle of the embedding   $i$ is trivial, hence the Todd class is equal to one).

4. Substituting the formulas obtained in items  2 and 3 of the proof in Eq.~\eqref{eq-atsi5}, we obtain
$$
  \ind \mathcal{D}'_0=p_*\bigl(i_*(\ch\sigma_+(\mathcal{A}) \Td(T_{\mathbb{C}}M\times_g \mathbb{S}^1) \bigr)=
p_{0*}\bigl(\ch\sigma_+(\mathcal{A})) \Td(T_{\mathbb{C}}M\times_g \mathbb{S}^1) \bigr),
$$
where $p_*$ and $p_{0*}$ stand for Gysin maps in cohomology   (integration over the fundamental cycle), induced by the projections  $p:T(M\times_g \mathbb{S}^1)\to pt$ and $p_0:SM\times_g\mathbb{S}^1\to pt$.

The proof of the lemma is complete.
\end{proof}

On the cylinder consider the vector bundle   $\im \sigma(P)\in  \Vect(S^*M\times[0,1])$ and
identify the fibers of this bundle over the components of the boundary using the mapping  $\sigma(D_0)$ (see \eqref{eq-2-term1}) as follows
\begin{equation}\label{eq-glue2}
\mathcal{V}=
\left\{
 \begin{array}{c}
   (x,\xi,t,v)\\
   v\in\im \sigma(P)(x,\xi)
 \end{array}
\right\}/
\bigl\{(x,\xi,0,v)\sim ((\partial g)^{-1}(x,\xi),1 ,\sigma(D_0)((\partial g)^{-1}(x,\xi))v)\bigr\}.
\end{equation}
This space is a vector bundle $\mathcal{V}\in \Vect(S^*M\times_g \mathbb{S}^1)$ over the torus  $S^*M\times_g \mathbb{S}^1$.
\begin{lemma}\label{lem-5}
One has an isomorphism of vector bundles over   $S^*M\times_g\mathbb{S}^1$:
\begin{equation}\label{eq-isomm1}
  \sigma_+(\mathcal{A} )\simeq  \mathcal{V}.
\end{equation}
\end{lemma}
\begin{proof}
The pull-backs of the bundles  $ \sigma_+(\mathcal{A} )$ and $ \mathcal{V}$ to the cylinder $S^*M\times[0,1]$
are equal to 
$$
\im p\left(\frac \pi 2 \chi(t)\right)\quad \text{and} \quad \im p(0),\qquad \text{where }p(\varphi)=\sigma(P(\varphi)).
$$
The formula
\begin{equation}\label{eq-mu1}
u_{\frac\pi 2\chi(t)}^{-1}: \im p\left(\frac\pi 2\chi(t)\right)\lra \im p(0), \qquad \text{where }u_\varphi=\sigma(U_\varphi)
\end{equation}
defines a vector bundle  isomorphism  on the cylindeer. This mapping is well defined by Eq.~\eqref{eq-unit1}.

Let us verify that  the isomorphism \eqref{eq-mu1} of vector bundles over the cylinder extends by continuity to an isomorphism of bundles on the torus.  To prove this, it suffices to show that the diagram 
\begin{equation}\label{eq-diag2}
 \xymatrix{
   \im p(\pi/2) \ar[r]^{u^{-1}_{\pi/2}} \ar[d] &  \im p(0) \ar[d] \\
     \im p(0) \ar[r]^{u^{-1}_{0}}  & \im p(0)
 }
\end{equation}
is commutative. Here the horizontal mappings are just the restrictions of the isomorphism   $u_{\frac\pi 2\chi(t)}^{-1}$  at $t=0$
(upper row)  and at $t=1$ (lower row), while the vertical mappings are just identifications of vector bundles on the boundary of the cylinde. Recall that these identification mappings are defined in \eqref{eq-glue1} and \eqref{eq-glue2}
and give bundles on the torus.

Let us prove that   \eqref{eq-diag2} is a commutative diagram. We have
$$
u_0=1_{2N},\quad
u^{-1}_{\pi/2}=\sigma
   \left(
    \begin{array}{cc}
      1_N-P & g^{-1*}(D_0^{-1}P)\\
      -g^{-1}(PD_0) & 1_N-g^{-1*}(P)
    \end{array}
   \right),
$$
$$
p(0)=\sigma\left(
    \begin{array}{cc}
       P & 0\\
     0 & 0
    \end{array}
   \right),\; p(\pi/2)=\sigma\left(
    \begin{array}{cc}
       0 & 0\\
     0 & g^{-1*}(P)
    \end{array}
   \right).
$$
Let $(x,\xi,0,v)\in \im p(\pi/2)$. This means that $(x,\xi)\in S^*M$ and  $v\in \im p((\partial g)^{-1}(x,\xi))$.
Then, passing the diagram  \eqref{eq-diag2} through the lower left corner, we obtain the elements  
\begin{equation}\label{eq-el1}
 \xymatrix{
   (x,\xi, 0,v) \ar[d] &    \\
   ((\partial g)^{-1}(x,\xi),v,0) \ar[r]  & (((\partial g)^{-1}(x,\xi),v).
 }
\end{equation}
If we now pass the diagram through the right upper corner, we obtain the elements  
\begin{equation}\label{eq-el2}
 \xymatrix{
   (x,\xi, 0,v) \ar[r] &   (x,\xi, {}^*\sigma(D_0)^{-1}((\partial g)^{-1}(x,\xi))v)  \ar[d]\\
   & ((\partial g)^{-1}(x,\xi),v).
 }
\end{equation}
Since the elements obtained in the right lower corner in   \eqref{eq-el1} and \eqref{eq-el2} are equal, the diagram  \eqref{eq-diag2} is commutative. Hence, \eqref{eq-mu1} defines an isomorphism on the cylinder, and this isomorphism defines the desired isomorphism   \eqref{eq-isomm1} on the torus.
\end{proof}

\subsection{Comparison of the topological indices of the $\psi$DO on the torus and of the original operator}

Consider the equalities
\begin{multline}\label{eq-top2}
\ind D= \ind \mathcal{D}'_0 =\int_{S^*M\times_g\mathbb{S}^1} \ch[ \sigma_+(\mathcal{A})]\Td(T^*_{\mathbb{C}}M\times_g \mathbb{S}^1)=\\
\int_{S^*M\times_g\mathbb{S}^1} \ch \mathcal{V}\Td(T^*_{\mathbb{C}}M\times_g \mathbb{S}^1)=
\int_{S^*M\times_g\mathbb{S}^1} \ch  \mathcal{V} \ch \psi(T^*_{\mathbb{C}}M\times_g \mathbb{S}^1)=\\
\int_{S^*M\times_g\mathbb{S}^1} \ch(\mathcal{V}\otimes {\psi(T^*_{\mathbb{C}}M\times_g\mathbb{S}^1 )}).
\end{multline}
Here the first equality follows from the results of Sec.~\ref{sec-5}. The second follows from Lemma~\ref{lem-torus1}
and third follows from Lemma~\ref{lem-5}. The fourth equality is just the definition of operation   $\psi$.

To complete the proof of index theorem~\ref{th-index1}, it suffices to show that the topological index  $\ind_{top} D$ of the operator~\eqref{eq-2-term1} is equal to  the right-hand side in \eqref{eq-top2}. Note that, generally speaking, the element   $\psi(T^*_{\mathbb{C}}M\times_g\mathbb{S}^1 )\in K(M\times_g\mathbb{S}^1)\otimes\mathbb{Q}$ is a virtual bundle, i.e., a linear combination of vector bundles with rational coefficients   (see Sec.~\ref{sec-4}). 
To simplify the notation, we shall assume that   $\psi(T^*_{\mathbb{C}}M\times_g\mathbb{S}^1 )$ is a vector bundle, thus omitting the corresponding sum and coefficients.

Let $\mathcal{F}=\mathcal{V}\otimes   \psi(T^*_{\mathbb{C}}M\times_g\mathbb{S}^1 ) $  for brevity.
The sections of   $\mathcal{F}$ are just sections $F=\{F(x,\xi,t)\}$ of the bundle $\im p\otimes\psi(T^*_\mathbb{C}M)$ (here $p=\sigma(P)$)
on the cylinder $S^*M\times[0,1]$   such that
\begin{equation}\label{eq-cond6}
\sigma(D_0)(\partial g)^*F|_{t=0}=F|_{t=1}.
\end{equation}
This statement is verified by a direct computation using Eq.~\eqref{eq-glue2}.

Let us compute the Chern character of~$\mathcal{F}$ using formalism of connections.
Let  $\nabla_p=p\nabla_{\psi}  p$ be a connection in the bundle $(\im p)\otimes \psi$.
Here  $\psi$ stands for the bundle $\psi(T^*_\mathbb{C}M)$ for brevity,
while $\nabla_\psi$ is a connection in this bundle. Then the formula
\begin{equation}\label{eq-connection1}
 \nabla'_{tor}=dt\frac\partial{\partial t}+t\nabla_p+(1-t)[\sigma(D_0)(\partial g)^*]^{-1}\nabla_p[\sigma(D_0)(\partial g)^*]
\end{equation}
defines the connection in $\mathcal{F}$:
$$
 \nabla'_{tor}: \Lambda(S^*M\times_g\mathbb{S}^1,\mathcal{F})\lra  \Lambda(S^*M\times_g\mathbb{S}^1,\mathcal{F}).
$$
A direct computation using \eqref{eq-cond6} shows that this connection is well defined. The Chern character form of 
 $\mathcal{F}$ is defined by the classical formula
\begin{equation}\label{eq-chern-ch6}
\ch\mathcal{F}= \tr \left(p  \exp\left(-\frac{{\nabla_{tor}'}^2}{2\pi i}\right)\right).
\end{equation}

On the other hand, consider the original operator  \eqref{eq-2-term1}. By Lemma~\ref{lem-useful1} we have for this operator  (and $E=\psi(T^*_\mathbb{C}M)$)
\begin{equation}\label{eq-qq78}
\ch_{\psi(T^*_\mathbb{C}M)}[\sigma(D)]=    \tr \left(   \exp\left(-\frac{{\nabla^2_{tor}}}{2\pi i}\right)\right),
\end{equation}
where the noncommutative connection is equal to
\begin{equation}\label{eq-qq43}
\nabla_{tor}= dt\frac\partial{\partial t}+t\nabla +(1-t)(\sigma(D )^{-1})\nabla (\sigma(D)) 
\end{equation}
and is expressed in terms of some connection   $\nabla$  in the bundle  $\mathbb{C}^N\otimes \psi$ over $S^*M$.
Let us now define $\nabla$ as
\begin{equation}\label{eq-expr1}
\nabla=p\nabla_\psi p+(1-p)\nabla_\psi (1-p)
\end{equation}
and recall that
\begin{equation}\label{eq-expr2}
 \sigma(D)=\sigma(D_0)Tp+(1-p),\quad  \sigma(D)^{-1}=(\sigma(D_0)T)^{-1}p+(1-p)
\end{equation}
(see \eqref{eq-2-term1}). Substituting the expressions  \eqref{eq-expr1} and \eqref{eq-expr2} in Eq.~\eqref{eq-qq43}, we obtain
\begin{multline*}
\nabla_{tor}= dt\frac\partial{\partial t}+tp\nabla_\psi p+
(1-p)\nabla_\psi(1-p)+(1-t)(\sigma(D_0)T)^{-1}p\nabla_\psi p (\sigma(D_0)T)=\\
= p\nabla'_{tor} p+(1-p)(dt \frac\partial{\partial t} +\nabla_\psi)(1-p).
\end{multline*}
This implies that the curvature form is equal to 
$$
(\nabla_{tor})^2= p{\nabla'_{tor}}^2 p+[(1-p)\nabla_\psi(1-p)]^2.
$$
Hence the Chern character forms for the connections   $\nabla'_{tor}$ and $\nabla_{tor}$ differ by a form that does not contain   $dt$. Hence, the integrals of these Chern character forms are equal:
$$
\int_{S^*M\times[0,1]}   \tr \left(p   \exp\left(-\frac{(\nabla_{tor})^2}{2\pi i}\right)\right)=
\int_{S^*M\times[0,1]}   \tr \left( p  \exp\left(-\frac{(\nabla'_{tor})^2}{2\pi i}\right)\right),
$$
i.e., we obtain the desired equality
$$
\int_{S^*M\times_g\mathbb{S}^1} \ch(\mathcal{V}\otimes {\psi(T^*_{\mathbb{C}}M\times_g\mathbb{S}^1 )})=
\int_{S^*M\times[0,1]}  \ch_{\psi(T^*_\mathbb{C}M)}[\sigma(D)]=\ind_{top}D.
$$

The proof of the index theorem~\ref{th-index1} is now complete.

\section{Index formula in cyclic cohomology}

In this section we give an interpretation of the index formula   \eqref{aaa} in terms of cyclic cohomology (see~\cite{Con1,Tsy1}
and for cyclic cohomology of crossed products \cite{BrNi1,GeJo1,Nest1}).

\subsection{Equivariant Chern character}

\paragraph{1. Chern character in cyclic cohomology.}  Let   $E\in\Vect(X)$ be a vector bundle
over a smooth closed oriented manifold   $X$, $\dim X=n$. We fix a connection $\nabla_E$ in $E$ and define following \cite{Gor1} the multilinear functionals
\begin{multline}\label{eq-jlo1a}
\Char^k(E,\nabla_E;a_0,a_1,...,a_k)=\\
=
\frac{(-1)^{(n-k)/2}}{((n+k)/2)!}
\sum_{i_0+i_1+\ldots+i_k=(n-k)/2} \int_X \tr_E
\bigl[\left( a_0 \theta^{i_0}\nabla(a_1)\theta^{i_1}\nabla(a_2)\ldots \nabla(a_k)\theta^{i_k} \right)_0\bigr]
\end{multline}
$k=n,n-2,n-4,\ldots$
 (cf. Jaffe--Lesnievski--Osterwalder formula~\cite{JLO1}). Here for a noncommutative form $\omega$ by  $\omega_0$
we denote the coefficient at $T^0=1$, $\theta=\nabla_E^2$ is the curvature of the connection,
while the operator $\nabla:\Lambda(X,\End E)_\mathbb{Z}\to \Lambda(X,\End E)_\mathbb{Z}$ is defined as
$$
\nabla (\omega)=\nabla\omega-(-1)^{\deg\omega}\omega\nabla
$$
or more explicitly
$$
\nabla (\sum_k\omega_k T^k)= \sum_k\bigl[
 \nabla_E\omega_k-(-1)^{\deg\omega_k}\omega_k g^{k*}(\nabla_E)\bigr]T^k,
$$
where the expression $ \nabla_E\omega_k-(-1)^{\deg\omega_k}\omega_k g^{k*}(\nabla_E)$ is an operator of multiplication 
by a $1$-form. It follows from  \cite{Gor1} that the collection of functionals $\{\Char^k(E,\nabla_E)\}$
defines a cyclic cocycle over the algebra $C^\infty(X,\End E)_\mathbb{Z}$,  and the class of this cocycle 
in periodic cyclic cohomology   
$$
\Char(E)=[\{\Char^k(E,\nabla_E)\}]\in HP^*(C^\infty(X,\End E)_\mathbb{Z})
$$
does not depend on the choice of connection $\nabla_E$.

\begin{example}\label{exaa2}
Let $E$ be a flat bundle, i.e., $\theta=0$. Then the Chern character \eqref{eq-jlo1a} has only one nonzero component  $ \Char^n(E)$ that is equal to
\begin{equation}\label{eq-flat2}
\Char^n(E,\nabla_E;a_0,...,a_n)=\frac 1{n!}\int_X\tr_E(a_0\nabla(a_1)\nabla(a_2)\ldots\nabla (a_n))_0,\quad \dim X=n.
\end{equation}
\end{example}

\paragraph{2. Relation to the Chern character in Haefliger cohomology.}

\begin{proposition}\label{prop-eq1}
One has a commutative diagram
\begin{equation}\label{eq-diaga3}
\xymatrix{
 &  K_0(C^\infty(X,\End E)_\mathbb{Z})\ar[ld]_\ch \ar[dr]^{\;\;\;C_n\langle\cdot, \Char(E)\rangle}&  \\
  H^{ev}(X/\mathbb{Z})\ar[rr]_{\int_X}& & \mathbb{C},
}
\end{equation}
where  $C_n=(2\pi i)^{-n/2}$,  $\ch$ is the Chern character from Subsec.~\ref{sec-2-1}, $\int_X$ stands for the integral, and $\langle \cdot,\cdot\rangle$ is the pairing of the $K_0$-group with cyclic cohomologies. This pairing is defined by the formula
\begin{equation}\label{eq-pairing1}
\langle[p],[\varphi]\rangle=
\sum_k\frac{(-1)^k(2k)!}{k!}\varphi_{2k}   (p-1/2,p,\ldots, p),
\end{equation}
where $[p]\in K_0(C^\infty(X,\End E)_\mathbb{Z}),[\varphi]\in HP^{ev}(C^\infty(X,\End E)_\mathbb{Z}),$ and
the cyclic cocycle $\varphi$ is extended to matrix elements in the usual way
$$
\varphi_{l}(m_0\otimes a_0,m_1\otimes a_1,\ldots, m_l\otimes a_l )=
\tr(m_0m_1\ldots m_l)\varphi_l(a_0,a_1,\ldots, a_l).
$$
\end{proposition}

\begin{proof}
1. Let us make an additional construction. Namely, we embed the triangle~\eqref{eq-diaga3} in the diagram\vspace{-5mm}
\begin{equation}\label{eq-diaga4}
\xymatrix{
  & & K_0(C^\infty(X,\Mat_N(\mathbb{C}))_\mathbb{Z})\ar[lldd]_\ch \ar[ddrr]^{C_n\langle\cdot, \Char(X\times{\mathbb{C}^N})\rangle}&  &\\
 &  & K_0(C^\infty(X,\End E)_\mathbb{Z})\ar[u]\ar[lld]^\ch \ar[drr]_{C_n\langle\cdot, \Char(E)\rangle}& & \\
  H^{ev}(X/\mathbb{Z})\ar[rrrr]_{\int_X}& & & &\mathbb{C}
}
\end{equation}
Here the vertical mapping
$ 
 K_0(C^\infty(X,\End E)_\mathbb{Z})\to K_0(C^\infty(X,\Mat_N(\mathbb{C}))_\mathbb{Z})
$ 
is induced by the embedding $E\subset X\times \mathbb{C}^N$ in the trivial bundle.

2. We claim that the left and the right triangles of the diagram~\eqref{eq-diaga4} are commutative.  
Indeed, let us prove the commutativity of the left triangle   (the commutativity of the right triangle is obtained similarly). Suppose that $E=\im q\subset X\times \mathbb{C}^N$, where $q$ is a projection in the trivial bundle. Then we have an isomorphism
$$
C^\infty(X,\End E)_\mathbb{Z}=q \bigl[C^\infty(X,\Mat_N(\mathbb{C}))_\mathbb{Z} \bigr]q.
$$
In particular, to  a projection $p$  over $C^\infty(X,\End E)_\mathbb{Z}$ we assign a projection $p'$ over   the algebra $C^\infty(X,\Mat_N(\mathbb{C}))_\mathbb{Z}$. Thus, we have
$$
\ch[p]=[\ch(p,\nabla_E)]=[\ch(p',q\nabla_E q+(1-q)d(1-q))]=\ch[p'].
$$
Here the first and last equalities follow from the definition of the Chern character, and the equality in the middle follows
from the equality of the corresponding differential forms.  

3. The perimeter of the diagram \eqref{eq-diaga4} is also a commutative triangle. Indeed, in the trivial bundle, let us choose the flat connection defined by the exterior differential   $d$. Then by Example~\ref{exaa1} for a projection  $p$ over $C^\infty(X,\Mat_N(\mathbb{C}))_\mathbb{Z}$ we obtain
$$
\int_X \ch[p]=\frac{1}{(n/2)!}\left(-\frac{1}{2\pi i}\right)^{n/2}\int_X \tr \left[p(dpdp)^{n/2}\right]_0.
$$
On the other hand, it follows from the formula obtained in the Example~\ref{exaa2} that 
$$
\langle[p],\Char(X\times{\mathbb{C}^N})\rangle=\frac{(-1)^{n/2}}{(n/2)!}\int_X \tr \left[p(dpdp)^{n/2}\right]_0.
$$
We see that the last two expressions differ only by the factor $C_n=(2\pi i)^{-n/2}$. This proves that the perimeter of the diagram  \eqref{eq-diaga4} is commutative.

4. In the diagram \eqref{eq-diaga4}, we proved the commutativity of all the triangles, except for the lower triangle. Hence, the lower triangle is commutative.   

The proof of the proposition is complete.
\end{proof}

\paragraph{3. Equivariant Chern character \cite{Gor1}.}
{\em The equivariant Chern character} of a $g$-bundle  $E$ on $X$
$$
\Ch(E)\in HP^*(C^\infty(X)\rtimes\mathbb{Z})
$$
is defined as
$$
\Ch(E):=\beta^*\Char(E) ,
$$
where  $\beta^*:HP^*(C^\infty(X,\End E)\rtimes\mathbb{Z}) \lra HP^*(C^\infty(X)\rtimes\mathbb{Z}) $
is the mapping induced by the homomorphism of algebras 
$$
\beta:C^\infty(X)\rtimes\mathbb{Z}\lra C^\infty(X,\End E)\rtimes\mathbb{Z};\quad
\sum_k\omega_k T^k\mapsto \sum_k(\omega_k\otimes 1_E)\widetilde{T}^k.
$$
There is an analogue of the commutative diagram~\eqref{eq-diaga3} for the equivariant Chern character. Namely, one has
\begin{equation}\label{eq-diaga5}
\xymatrix{
 &  K_0(C^\infty(X)\rtimes\mathbb{Z})\ar[ld]_{\ch_E} \ar[dr]^{\;\;\;C_n\langle\cdot, \Ch(E)\rangle}&  \\
  H^{ev}(X/\mathbb{Z})\ar[rr]_{\int_X}& & \mathbb{C}.
}
\end{equation}

\subsection{Index formula in cyclic cohomology}

Given a $g$-bundle $E\in\Vect(X)$ over a smooth closed oriented manifold $X$, we define the equivariant Todd class 
$$
\td(E)\in HP^*(C^\infty(X)\rtimes\mathbb{Z})
$$
as $\td(E):=\Ch(\psi(E)),$ where $\psi$ is the operation in rational $K$-theory defined in Sec.~\ref{sec-4}.

\paragraph{1. Index formula.}

\begin{theorem}
For an elliptic operator $D$ one has an index formula
\begin{equation}\label{eq-ind-forla1}
 \ind D=(2\pi i)^{-n}\langle [\sigma(D)],\td (\pi^*T^*_\mathbb{C}M)\rangle,\qquad \dim M=n,
\end{equation}
where $\pi: S^*M\times\mathbb{S}^1\lra M$ is the projection and the brackets $\langle,\rangle$
stand for the pairing of $K$-theory with cyclic cohomology (see \eqref{eq-pairing1}).
\end{theorem}
\begin{proof}
The index formula \eqref{aaa} gives us
\begin{equation}\label{bbb9}
\ind D=\int_{S^*M\times\mathbb{S}^1} \ch_{\psi(\pi^*T^*_\mathbb{C}M)}[\sigma(D)].
\end{equation}
Using the commutative diagram \eqref{eq-diaga5} and the definition of the equivariant Todd class, we can rewrite 
the right-hand side in~\eqref{bbb9} in the desired form
\begin{multline*}
\int_{S^*M\times\mathbb{S}^1} \ch_{\psi(\pi^*T^*_\mathbb{C}M)}[\sigma(D)]=
(2\pi i)^{-n}\langle [\sigma(D)],\Ch(\psi(\pi^*T^*_\mathbb{C}M))\rangle=\\
= (2\pi i)^{-n}\langle [\sigma(D)],\td( \pi^*T^*_\mathbb{C}M )\rangle.
\end{multline*}
\end{proof}

\paragraph{2.  A special case.}   Suppose that the Todd class $\Td(T^*_\mathbb{C}(M\times_g\mathbb{S}^1))$ of the complexification of the cotangent bundle of the twisted torus is equal to one. Then it turns out that in this case one can replace the class   $\td (\pi^*T^*_\mathbb{C}M)$ in \eqref{eq-ind-forla1} simply by the transverse fundamental class of the manifold   $S^*M\times \mathbb{S}^1$  in the sense of \cite{Con8}. This  enables one to write the index formula in the form
\begin{equation}\label{eq-nice-index1}
\ind D=\frac{(n-1)!}{(2\pi i)^n(2n-1)!}\int_{S^*M} \tr (\sigma^{-1}d\sigma)^{2n-1}_0,\qquad \sigma=\sigma(D).
\end{equation}
The index formula \eqref{eq-nice-index1} is a corollary of the following more general statement.
\begin{proposition}\label{prop-crutial1}
Suppose that in the homology class Poincar\'e dual to the Todd class   $ \Td(T^*_\mathbb{C}M\times_g\mathbb{S}^1)$
there exists a representative\footnote{Here the homology group is treated in terms of closed de Rham currents  (see \cite{Rha1}).} of the form
$$
\omega \longmapsto z\left( \int_{\mathbb{S}^1} \omega\right),\quad
\omega\in \Lambda(S^*M\times_g\mathbb{S}^1),
$$
where $z$ is a closed $g$-invariant current on $S^*M$. Then the equivariant Todd class in the index formula   \eqref{eq-ind-forla1} can be replaced by the collection of cyclic cocycles with the components  
$$
(a_0,...,a_{2k})\longmapsto\frac{(2\pi i)^{n-k}}{(2k)!}z\left(a_0da_1 da_2...da_{2k}\right),\qquad k=0,1,...,n.
$$
\end{proposition}

\begin{proof}
The proof of this proposition is similar to the proof of index formula   \eqref{eq-ind-forla1}. The main difference is that instead of equality \eqref{eq-top2}  one uses equalities of the form
\begin{equation}\label{eq-top2z}
\ind D= \ind \mathcal{D}'_0 = 
\int_{S^*M\times_g\mathbb{S}^1} (\ch \mathcal{V})\Td(T^*_{\mathbb{C}}M\times_g \mathbb{S}^1)=
z\left(\int_{\mathbb{S}^1} \ch  \mathcal{V}\right).
\end{equation}
\end{proof}

\begin{example}Suppose that $ \Td(T^*_\mathbb{C}M\times_g\mathbb{S}^1)=1$. Then we can take the current $z$ of degree $2n-1$ defined by integration over $S^*M$.   Then Proposition~\ref{prop-crutial1} gives the  index formula \eqref{eq-nice-index1}  (after a standard integration over $\mathbb{S}^1$). This remark applies, for instance, in the case  of elliptic operators for a  diffeomorphism of the sphere in the connected component of the identity (see \cite{CoMo2,Mos2}).  
\end{example}

Proposition~\ref{prop-crutial1} can be applied if $g$ is an isometry. In this case we define  the current   $z$  by the formula
$$
z(\omega)=\int_{S^*M} \omega \wedge \Td(T^*_\mathbb{C}M),\quad \omega\in \Lambda(S^*M),
$$
where $ \Td(T^*_\mathbb{C}M)$ is the differential form representing the Todd class using a $g$-invariant metric. In this case, we obtain the index formula first proved in~\cite{NaSaSt17}.

\section{Examples. Remarks}

\subsection{Example. Operators on the torus $\mathbb{T}^3$}

The index formula \eqref{eq-ind-forla1}, despite its compact and elegant form, often leads to serious computational difficulties, when one really needs to compute the index of a specific operator.  To solve this problem, it is sometimes useful to simplify the formula so that the simplified formula could really be used to compute the desired number.

In this subsection, we exibit a procedure of this form for a relatively simple operator related to
the Dirac operator. The answer we obtain is quite suitable to obtain explicit numerical expression for the index of the problem.  

1.~Consider the torus $\mathbb{T}^3=\mathbb{R}^3/2\pi\mathbb{Z}^3$
with coordinates  $x=(x_1,x_2,x_3)$ and the diffeomorphism\footnote{This diffeomorphism 
is the ``Arnold's cat map'' \cite{Arn1} acting along  $x_1,x_2$ and the identity map along $x_3$.}
$$
g: \mathbb{T}^3\lra \mathbb{T}^3, \quad
g\left( \begin{matrix}   x_1\\    x_2\\x_3 \end{matrix} \right)=\left(
 \begin{matrix}
   2 & 1 & 0\\
   1 & 1 & 0\\
   0 & 0 & 1
 \end{matrix}
\right)
\left( \begin{matrix}   x_1\\    x_2\\x_3 \end{matrix} \right).
$$
Consider the Dirac operator on $\mathbb{T}^3$ 
\begin{equation}\label{eq-dir1}
 \sum_{j=1}^3 c_j \left(-i\frac\partial{\partial x_j}\right):
C^\infty(\mathbb{T}^3,\mathbb{C}^2)\lra C^\infty(\mathbb{T}^3,\mathbb{C}^2),
\end{equation}
where 
$$
 c_1=\left(
 \begin{matrix}
   0 & 1  \\
  1 & 0 \\
  \end{matrix}
\right),\;
c_2=\left(
 \begin{matrix}
   0 & -i  \\
  i & 0 \\
  \end{matrix}
\right),\;
c_3=\left(
 \begin{matrix}
   1 & 0  \\
  0 & -1 \\
  \end{matrix}
\right)
$$ 
stand for Pauli matrices. The Dirac operator is elliptic and self-adjoint in the space $L^2$. 
Therefore, it has a discrete real spectrum, while the eigenvalues have finite multiplicities. Consider the positive
spectral projection for the Dirac operator, i.e., the orthogonal projection on the subspace generated by eigenfunctions of the Dirac operator s with positive eigenvalues. This projection is denoted by $P$ and is  a $\psi$DO (see \cite{See5})  of order zero.

\begin{theorem}\label{th-55}
Let $f\in C^\infty(\mathbb{T}^3,\Mat_N(\mathbb{C}))$ be a function  ranging in invertible matrices. Then the operator
\begin{multline}\label{eq-nashop4}
D=(f\otimes 1) (1\otimes P)T(1\otimes P)+1\otimes (1-P):\\
H^s(\mathbb{T}^3,\mathbb{C}^{N}\otimes \mathbb{C}^2)\lra
H^s(\mathbb{T}^3,\mathbb{C}^{N}\otimes \mathbb{C}^2),
\end{multline}
where  $T=g^*$ is the shift operator for $g$, is Fredholm for all   $s$ and its index is equal to 
\begin{equation}\label{eq-index7}
\ind D= \frac 1 {(2\pi i)^2 3!}\int_{\mathbb{T}^3} \tr (f^{-1}df)^{3}.
\end{equation}
\end{theorem}
Let us write the operator \eqref{eq-nashop4} simply as
$$
 D=fPTP+1-P
$$
omitting the tensor products.

2.~Let us prove that $D$ is elliptic. To this end, we first compute the symbol of $P$. The symbol of the Dirac operator \eqref{eq-dir1} is equal to
$$
c(\xi)=c_1\xi_1+c_2\xi_2+c_3\xi_3\in \Mat_2(\mathbb{C}),
$$
where $\xi=(\xi_1,\xi_2,\xi_3)$ stand for variables dual  to  $x$.
In what follows, it is useful to write the following Clifford identity (e.g., see  \cite{LaMi1}):
\begin{equation}\label{eq-clifford}
 c(\xi)c(\xi')v+c(\xi')c(\xi)v=2(\xi,\xi')v, \quad \xi,\xi'\in\mathbb{R}^3,v\in\mathbb{C}^2,
\end{equation}
where in the right-hand side of (\ref{eq-clifford}) we have inner product of vectors.  
In particular,  Eq.~\eqref{eq-clifford} implies that the matrix
$$
p(\xi)=\frac{ {1}+c(\xi)}2,\quad |\xi|=1
$$
is a rank one projection. Moreover, this projection is just the positive spectral projection of the symbol  $c(\xi)$ of the Dirac operator. We extend this function  to a degree zero homogeneous function in  $\xi$. Then the results of the paper 
\cite{See5} give the equality
$$
\sigma(P)=p.
$$
We are now ready to prove that $D$ is elliptic. Consider the mapping
\begin{equation}\label{eq-uu1}
u(\xi)=p(\xi):\im (\partial g)^*p(\xi)\lra \im p(\xi).
\end{equation}
We claim that the mapping (\ref{eq-uu1}) is invertible (cf. \cite{Sav9}).  Indeed, let us consider the converse, i.e., suppose that for some   $\xi$ we have a nonzero vector
\begin{equation}
\label{aaa6}
v\in\im (\partial g)^*p(\xi)=\im p(g^{-1}\xi)  \quad \text{such that } 
p(\xi)v=0.
\end{equation}
In terms of Clifford multiplication, condition (\ref{aaa6})  is written  as: 
$$
c(\xi)v=-v,\quad  c\left(\frac{g^{-1}\xi}{|g^{-1}\xi|}\right) v=v.
$$
Substituting these two formulas in \eqref{eq-clifford}, we get
\begin{equation}
\label{aaaaa}
-2v=2\left(\xi,\frac{g^{-1}\xi}{|g^{-1}\xi|}\right )v\quad \text{or  } 
\cos(\xi,g^{-1}\xi)=-1,
\end{equation}
i.e., the vectors   $\xi$ and $g^{-1}\xi$ form an angle equal to $\pi$. But this can not be true, since the matrix  $g^{-1}$ has no negative eigenvalues.  
This contradiction shows that the mapping   \eqref{eq-uu1} is an isomorphism.

Denote by $U^{-1}$ a $\psi$DO  on $\mathbb{T}^3$ such that the restriction of its symbol to the subspace   $\im p(\xi)$
coincides with $u(\xi)^{-1}$. We claim that the operator  
$$
B= T^{-1}f^{-1}U^{-1}P+1-P
$$
is an almost inverse of   $D$ (i.e., inverse up to operators of negative order). Indeed, for example, let us compute the composition of symbols:
\begin{multline}\label{ugly1}
\sigma(D)\sigma(B)=(fpTp+1-p)(T^{-1}u^{-1}pf^{-1}+1-p)=\\
=fpTpT^{-1}u^{-1}pf^{-1}+(1-p)+(1-p)T^{-1}u^{-1}pf^{-1}=\\
=fp[(\partial g)^*p] u^{-1}p f^{-1}+(1-p)+(1-p)T^{-1} T p T^{-1}u^{-1}p f^{-1}=\\
=fpf^{-1}+(1-p)+(1-p)pT^{-1}u^{-1}p f^{-1}=p+1-p=1.
\end{multline}
The equality  $\sigma(B)\sigma(D)=1$ is obtained similarly.

Thus,  $D$ is elliptic and  Fredholm by Theorem~\ref{th-finit1}.

3.~By the index theorem~\ref{th-index1} the analytic index of $D$  is equal to the topological index of its symbol. 
Let us compute the topological index of $\sigma(D)$. The symbol $\sigma(D)$ has the factorization  
\begin{equation}\label{eq-deco6}
\sigma(D)=\sigma_0 \sigma_1,\quad \sigma_0=pfp+1-p,\quad \sigma_1=pTp+(1-p),
\end{equation}
into two elliptic symbols, where  $\sigma_0$ does not contain shift $T$.
Let us compute the topological indices of these symbols. The index of    $\sigma_0$ coincides with the 
Atiyah--Singer topological index and is equal to   (see \cite{BaDo1,BaDo2})
$$
\ind_{top} \sigma_0=\int_{S^*\mathbb{T}^3}\ch [f]\ch(\im p) \Td(T^*_\mathbb{C}\mathbb{T}^3)   ,
$$
where $[f]\in K^1(\mathbb{T}^3)$ is the class of   $f$ in the odd $K$-group. Further, we get
$$
\int_{S^*\mathbb{T}^3}\!\ch [f]\ch(\im p) \Td(T^*_\mathbb{C}\mathbb{T}^3)\!=\!
\int_{S^*\mathbb{T}^3}\!\ch [f]\ch(\im p) =\!\int_{\mathbb{T}^3}\!\ch [f]\int_{\mathbb{S}^2} \ch (\im p)  =C\!\int_{\mathbb{T}^3}\! \tr (f^{-1}df)^{3},
$$
where $C=((2\pi i)^2 3!)^{-1}$. Here we first noted that the tangent bundle of the torus is trivial and replaced the Todd class by one. Then, we used the decomposition   $S^*\mathbb{T}^3=\mathbb{T}^3\times \mathbb{S}^2$ with the coordinates $x,\xi$ on the factors. Moreover, since $f$ depends only on $x$, and $p$ depends only on $\xi$, the integral over $\mathbb{T}^3\times \mathbb{S}^2$ is just the product of an integral over $\mathbb{T}^3$ and an integral over $\mathbb{S}^2$.  In the next to the last equality, the Chern character in the first factor is represented by a differential form and we noted that   $\im p$ is the Bott bundle on  $\mathbb{S}^2$ (see \cite{Ati2} and \cite{Gil1}) and one has 
$$
\int_{\mathbb{S}^2} \ch \im p=1.
$$

So, we obtain
\begin{equation}\label{eq-d00}
\ind_{top} \sigma_0=\frac{1}{(2\pi i)^2 3!}\int_{\mathbb{T}^3} \tr (f^{-1}df)^{3}.
\end{equation}
\begin{proposition}\label{prop-6a}
One has $\ind_{top} \sigma_1=0$.
\end{proposition}
\begin{proof}
By Eq.~\eqref{eq-nice-index1}, the topological index of    $\sigma_1$ is equal to
\begin{equation}\label{eq-nicce1}
\ind_{top} \sigma_1=\frac{2!}{(2\pi i)^3 5!}\int_{S^*M} \tr (\sigma^{-1}_1d\sigma_1)^5_0.
\end{equation}
Let us compute this integral. The symbol $ \sigma_1\in C^\infty(S^*\mathbb{T}^3,\Mat_2(\mathbb{C}))\rtimes \mathbb{Z}$  is constant in $x$. Thus, its differential $d \sigma_1\in\Lambda^1(S^*\mathbb{T}^3,\Mat_2(\mathbb{C}))\rtimes \mathbb{Z}$ does not contain differentials $dx_j$. The product $ \sigma_1^{-1}d \sigma_1$ also has no differentials $dx_j$. This uses the fact that $g$ is a linear diffeomorphism. The same reasoning shows that the form
$$
(\sigma^{-1}_1d\sigma_1)^5
$$
(see \eqref{eq-nicce1}) also does not contain differentials $dx_j$. On the other hand, the degree of this form is equal to five. Therefore, this form is identically zero. Thus,   \eqref{eq-nicce1} implies that the topological index of  $\sigma_1$ is zero.
\end{proof}

The formula \eqref{eq-index7} now follows from Eq.~\eqref{eq-d00} and Proposition~\ref{prop-6a}. This completes the proof of Theorem~\ref{th-55}.

4.~Let us give a direct proof of Theorem~\ref{th-55}.

Namely, let us first prove that   $D$ is elliptic. One has an equality  (cf.  \eqref{eq-deco6}) modulo compact operators
$$
D=D_0D_1,\quad D_0=fP+(1-P),\quad D_1=PTP+(1-P).
$$
Here  $D_0$ is an elliptic  $\psi DO$  and its index (computed by the Atiyah--Singer formula)
is equal to the right-hand side in~\eqref{eq-d00} (see the above computation).
Thus, to prove Theorem~\ref{th-55}, it suffices to show that $D_1$ is a Fredholm operator of index zero.  
\begin{proposition}
The operator $D_1=PTP+(1-P)$ is invertible.
\end{proposition}
\begin{proof}
Let us represent functions on the torus as Fourier series $\sum_k a_k e^{i(k,x)}$,
where $k=(k_1,k_2,k_3)\in \mathbb{Z}^3$,  and  $(k,x)=k_1x_1+k_2x_2+k_3x_3$.
In this notation, we have
\begin{equation}\label{eq-exact1}
P\left(\sum_k a_k e^{i(k,x)}\right)=\sum_{k\ne 0} a_k p(k) e^{i(k,x)},
\end{equation}
where $p(k)$ is the value of the function $p(\xi)$ at $\xi=k$. In this representation, the shift operator is equal to
$$
T\left(\sum_k a_k e^{i(k,x)}\right)=\sum_{k} a_k  e^{i(gk,x)},
$$
since $g$ has a symmetric matrix.  Let $U^{-1}$ be   the operator defined by a formula of the form  \eqref{eq-exact1} using the symbol  $u^{-1}(\xi)$. In this case the mappings 
$$
P:\im TPT^{-1}\to \im P
$$
and
$$
U^{-1}:\im P\to \im TPT^{-1}
$$
are inverse to each other. It follows that the operator $B=T^{-1}U^{-1}P+1-P$ is the inverse of $D_1$. 
The proofs of these statements are similar to the computation   \eqref{ugly1}.
\end{proof}

\subsection{Remark. Special two-term operators as operators in subspaces}

Let us give here a method   of computing  the index of special two-term operators using elliptic theory in
subspaces defined by pseudodifferential projections  (see \cite{SaSt1,SaSt2}).  We write a special two-term operator~\eqref{eq-kvese2} as
\begin{equation}\label{eq-8}
D =QD_0TP+(1-Q)D_1(1-P): C^\infty(M,\mathbb{C}^N)\lra  C^\infty(M,\mathbb{C}^N).
\end{equation}
Without loss of generality, we can assume that $P$ and $Q$ are projections $P^2=P$, $Q^2=Q$.
In this case, the operator $D$ is a direct sum 
$$
D=QD_0TP\oplus (1-Q)D_1(1-P)
$$
of operators acting in subspaces  defined by the projections $P,Q,1-P,1-Q$. Using this decomposition, we can compute the index of 
  $D$. Indeed, we get
\begin{multline}\label{eq-deco11}
\ind D= \ind[QD_0TP:\im P\to \im Q]+\ind[(1-Q)D_1(1-P):\im (1-P)\to\im (1-Q)]=\vspace{1mm}\\
=  \ind[D_0 :\im g^* P\to \im Q]+\ind[D_1:\im (1-P)\to\im (1-Q)].
\end{multline}
Here in the last equality we used the fact that  $T$ defines an isomorphism of the ranges of projections  $P$ and $g^*P=TPT^{-1}$.

An application of  the index formulas obtained in the papers \cite{SaSt1,SaSt2} to the operators in \eqref{eq-deco11} gives an index formula for $D$. To formulate the result, consider the involution $\alpha:T^*M\to T^*M$, $\alpha(x,\xi)=(x,-\xi)$ 
and for a $\psi $DO $A$ let $\alpha^*A$ denote any $\psi$DO with the symbol $\alpha^*\sigma(A)$.
\begin{proposition}
Let a special two-term operator \eqref{eq-8} be elliptic, the manifold $M$ be odd-dimensional, and the projections $P,Q$ be even, i.e., they satisfy the condition
$$
\alpha^*\sigma(P) =\sigma(P ),\quad \alpha^* \sigma(Q) =\sigma(Q).
$$
Then one has an equality
\begin{equation}\label{eq-inda7}
\ind D=\frac 1 2  \ind \left[  D_0( \alpha^* (D_0)^{-1}) Q +
 D_1  (\alpha^* (D_1 )^{-1})(1 -Q)\right],
\end{equation}
where the operator in the square brackets is an elliptic $\psi$DO on   $M$.
\end{proposition}
\begin{proof}
1. Application of the index formula from the paper~\cite{SaSt1} to the operators $D_0 :\im g^* P\to \im Q$
and $D_1:\im (1-P)\to\im (1-Q)$ gives us
\begin{equation}\label{eq-sub1}
\ind (D_0 :\im g^* P\to \im Q)=\frac 1 2 \ind \left[  D_0( \alpha^* (D_0)^{-1}) Q +(1-Q)\right]+d(g^*P)-d(Q),
\end{equation}
\begin{multline}\label{eq-sub2}
\ind (D_1 :\im (1-P)\to \im (1-Q))=\\
=\frac 1 2 \ind \left[ Q+ D_1( \alpha^* (D_1)^{-1})  (1-Q)\right]+d(1-P)-d(1-Q),
\end{multline}
where $d$ is the  homotopy invariant of even pseudodifferential projections constructed in~\cite{SaSt1}.
Adding the last two expressions~\eqref{eq-sub1} and \eqref{eq-sub2}, we obtain the following expression for the index of $D$:
\begin{multline}\label{eq-medi1}
\ind D=\frac 1 2 \ind \left[  D_0( \alpha^* (D_0)^{-1}) Q +
 D_1  (\alpha^* (D_1))^{-1})(1 -Q)\right]+\\
+(d(g^*P)+d(1-P))-(d(Q)+d(1-Q)).
\end{multline}
The cited paper contains the following properties of the functional $d$: 
$$
 d(Q)+d(1-Q)=0\quad \text{and}\quad d(g^*P)=d(P).
$$  
Hence, the last two terms in Eq.~\eqref{eq-medi1}  are equal to zero and we obtain the desired index formula   \eqref{eq-inda7}.
\end{proof}
There is an analog of this proposition for so-called odd projections on even-dimensional manifolds   (see \cite{SaSt2}).

\subsection{Remark. A generalization of the notion of ellipticity}

In \cite{AnLe1,AnLe2}, a different condition of ellipticity of operators~\eqref{eq-oper1} is used. This condition does not require that the number of nonzero components of the inverse symbol is finite. In this situation, the symbol   is naturally   an element of the $C^*$-crossed product  $C(S^*M)\rtimes\mathbb{Z}$ (see~\cite{Zell1}) of the algebra of continuous symbols on  $S^*M$ by the action of the diffeomorphism $g$, and the ellipticity is just the invertibility in this $C^*$-crossed product. On the other hand, it was shown in the papers  \cite{SaSt25,Sav9} that an elliptic operator in this sense is stably homotopic to an operator elliptic in the sense of Definition~\ref{def-ell-1}. This implies that to obtain an index formula for this class of operators, it suffices to extend the cyclic cocycle   $\td\in HP^*(C^\infty(S^*M\times\mathbb{S}^1)\rtimes \mathbb{Z})$, see~\eqref{eq-ind-forla1} to some local algebra $\mathcal{A}$ such that
$$
 C^\infty(S^*M)\rtimes \mathbb{Z}\subset\mathcal{A}\subset C(S^*M)\rtimes \mathbb{Z}
$$ 
or, in more invariant form, to define a class $\overline{\td}\in HP^*(\mathcal{A})$ that is the pull-back of the class $\td\in HP^*( C^\infty(S^*M\times\mathbb{S}^1)\rtimes \mathbb{Z})$ under the embedding $ C^\infty(S^*M\times\mathbb{S}^1)\rtimes \mathbb{Z}\subset  \mathcal{A}$.

Such extentions are known for many interesting  classes of diffeomorphisms, for example see \cite{Bost1,Con8,Schwe1,SaSt25}.
Therefore,  we obtain an index formula of the type   \eqref{eq-ind-forla1} for operators elliptic in the sense of \cite{AnLe1}  for these classes of diffeomorphisms.


\hfill {\em Hannover-Moscow}

\end{document}